\PassOptionsToPackage{table}{xcolor}
\documentclass[onefignum,onetabnum,a4paper]{siamart190516}
\pdfoutput=1

\usepackage[utf8]{inputenc}
\usepackage[english]{babel}
\usepackage{amsmath, amsfonts, dsfont, gensymb, physics, upgreek, siunitx}
\usepackage{tikz, pgfplots, graphicx, epstopdf, float, hyperref, xcolor}
\usepackage[caption=false]{subfig}
\usepackage{enumitem, verbatim, listings}
\usepackage{mathtools, stmaryrd}
\usepackage{booktabs, array, ragged2e}
\usepackage[export]{adjustbox}
\usepackage{xspace}

\usetikzlibrary{trees}
\graphicspath{{./figures}}
\definecolor{paired1}{HTML}{a6cee3}
\definecolor{paired2}{HTML}{1f78b4}
\definecolor{paired3}{HTML}{b2df8a}
\definecolor{paired4}{HTML}{33a02c}
\definecolor{paired5}{HTML}{fb9a99}
\definecolor{paired6}{HTML}{e31a1c}
\definecolor{paired7}{HTML}{fdbf6f}
\definecolor{paired8}{HTML}{ff7f00}
\definecolor{paired9}{HTML}{cab2d6}
\definecolor{paired10}{HTML}{6a3d9a}

\hypersetup{
    colorlinks,
    linkcolor={red!50!black},
    citecolor={blue!50!black},
    urlcolor={blue!80!black}
}

\pgfplotsset{
  log x ticks with fixed point/.style={
      xticklabel={
        \pgfkeys{/pgf/fpu=true}
        \pgfmathparse{exp(\tick)}%
        \pgfmathprintnumber[fixed relative, precision=3]{\pgfmathresult}
        \pgfkeys{/pgf/fpu=false}
      }
  },
  log y ticks with fixed point/.style={
      yticklabel={
        \pgfkeys{/pgf/fpu=true}
        \pgfmathparse{exp(\tick)}%
        \pgfmathprintnumber[fixed relative, precision=3]{\pgfmathresult}
        \pgfkeys{/pgf/fpu=false}
      }
  }
}

\newcommand{\reals}{\mathbb{R}}

\newcommand{\md}{\mathrm{d}}
\renewcommand{\div}{\mathrm{div}}
\newcommand{\bigo}[1]{\mathcal{O}(#1)}

\newcommand{\jump}[1]{\left\llbracket #1 \right\rrbracket}
\newcommand{\avg}[1]{\left\{ #1 \right\}}

\renewcommand{\vec}[1]{\mathbf{#1}}

\DeclareMathOperator{\diag}{diag}
\DeclareMathOperator{\spn}{span}

\newcommand{\eye}{{\mathds{1}}}
\newcommand{\kron}{\otimes}

\newcommand{\be}{\vec{e}}
\newcommand{\bu}{\vec{u}}
\newcommand{\bv}{\vec{v}}
\newcommand{\bw}{\vec{w}}
\newcommand{\bn}{\vec{n}}

\renewcommand{\bf}{\vec{f}}

\newcommand{\bx}{\vec{x}}

\newcommand{\bB}{\vec{B}}

\newcommand{\uf}{\underline{f}}
\newcommand{\ug}{\underline{g}}
\newcommand{\up}{\underline{p}}
\newcommand{\uq}{\underline{q}}
\newcommand{\ur}{\underline{r}}
\newcommand{\uu}{\underline{u}}

\newcommand{\ubu}{\underline{\bu}}

\newcommand{\ubf}{\underline{\vec{f}}}

\newcommand{\Jac}{\mathrm{D}F}
\newcommand{\mesh}{\mathcal{T}_h}
\newcommand{\betah}{\tilde{\beta}}

\title{A scalable and robust vertex-star relaxation for
high-order FEM\thanks{Submitted to the editors August 31, 2021.
      \funding{
         PDB was supported by the University of Oxford Mathematical Institute Graduate Scholarship.
         PEF was supported by EPSRC grants EP/V001493/1 and EP/R029423/1.
         We would also like to thank Lawrence Mitchell for his helpful advice 
         for the implementation in Firedrake. 
      }
   }
}

\author{Pablo D.\ Brubeck\thanks{
Mathematical Institute,
University of Oxford,
Oxford UK (\email{brubeckmarti@maths.ox.ac.uk})
}
\and Patrick E.\ Farrell\thanks{
Mathematical Institute,
University of Oxford,
Oxford UK (\email{patrick.farrell@maths.ox.ac.uk})
}
}

\headers{Sparse vertex-star relaxation for high-order FEM}{P.~D.~Brubeck and P.~E.~Farrell}

\begin{document}

\numberwithin{equation}{section}
\maketitle

\begin{abstract}
Pavarino proved that the additive Schwarz method with vertex patches and a
low-order coarse space gives a $p$-robust solver for symmetric and coercive
problems \cite{pavarino93}. However, for very high polynomial degree it is
not feasible to assemble or factorize the matrices for each patch.  In this
work we introduce a direct solver for separable patch problems that scales
to very high polynomial degree on tensor product cells.  The solver
constructs a tensor product basis that diagonalizes the blocks in the
stiffness matrix for the internal degrees of freedom of each individual
cell. As a result, the non-zero structure of the cell matrices is that of
the graph connecting internal degrees of freedom to their projection onto
the facets. In the new basis, the patch problem is as sparse as a low-order
finite difference discretization, while having a sparser Cholesky
factorization.  We can thus afford to assemble and factorize the matrices
for the vertex-patch problems, even for very high polynomial degree. In the
non-separable case, the method can be applied as a preconditioner by
approximating the problem with a separable surrogate. 
We demonstrate the approach by solving the Poisson equation 
and a $H(\mathrm{div})$-conforming interior penalty discretization of 
linear elasticity in three dimensions at $p = 15$.
\end{abstract}



\begin{DOI}
10.1137/21M1444187
\end{DOI}

\section{Introduction}

For problems with smooth solutions, high-order finite element methods offer
very good convergence properties, and in some cases they do not exhibit locking
phenomena found in low-order methods.  Moreover, there exist optimal
matrix-free algorithms for operator evaluation with high arithmetic intensity,
arising from data locality, that make
efficient use of modern parallel hardware architectures. 
Unfortunately, the conditioning of the stiffness
matrix is severely affected by the polynomial degree of the approximation.  In
order to obtain practical iterative solvers, we require good preconditioners.

Optimal solvers are often obtained from a multiplicative multigrid V-cycle
where the smoother consists of a domain decomposition method, such as additive
Schwarz with a particular space decomposition. The multigrid algorithm is then
accelerated by a Krylov subspace method, such as preconditioned conjugate
gradients (PCG). The choice of space decomposition in the relaxation is crucial
for robustness with respect to the cell size $h$, the polynomial degree $p$,
and parameters in the equation.

One of the cheapest relaxations, with $\bigo{p^d}$ computational cost, is
diagonal scaling, also known as point-Jacobi.  The diagonally preconditioned
Laplacian has a condition number of $\bigo{p^{2d-2}}$ \cite{maitre96}. This
implies that the number of PCG iterations, and therefore the number of residual
evaluations, is $\bigo{p^{d-1}}$, incurring a total cost of $\bigo{p^{2d}}$.
In order to minimize the time to solution,
it is reasonable to consider more expensive relaxation methods that converge in
fewer iterations. Ideally, we wish to balance the cost of applying the
relaxation with that of updating the residual.  On tensor product elements, the
latter operation can be done quickly in $\bigo{p^{d+1}}$ operations via the
sum-factorization. Sum-factorization breaks down the residual evaluation into
products of one-dimensional operators and diagonal scalings \cite{orszag80}. 

In 1993, Pavarino proved that the additive Schwarz method with a
vertex-centered space decomposition and an additive coarse space of
lowest-order ($p=1$) gives a robust solver with respect to $h$ and $p$ for
symmetric and coercive problems \cite{pavarino93}.  This type of space
decomposition is often referred to as generous overlap and is illustrated in
Figure \ref{fig:patches}.  We use the terminology of \cite{pcpatch} and refer
to the subdomains in this space decomposition as vertex-star patches, as they
are constructed by taking all the degrees of freedom (DOFs) on the topological
entities in the \textit{star} of each vertex.

The most straightforward implementation of a vertex-star solver involves the
assembly and direct factorization of the $\bigo{p^d} \times \bigo{p^d}$ patch
matrices (which are dense for Lagrange shape functions).  This becomes
prohibitively expensive at very high polynomial degrees, with the Cholesky
factorization of such a matrix requiring $\bigo{p^{3d}}$ operations. However, 
there exist bases for which the element matrices are sparse on affine cells,
such as the hierarchical Lobatto shape functions \cite{szabo91}. 
In this basis, the stiffness matrix has a 5-point stencil in 2D, and a much
larger 13-point stencil in 3D.

Efficient relaxation methods that are $p$-robust may arise from the
discretization of an auxiliary problem for which fast inversion techniques are
available. For a more general approach to auxiliary space techniques, we refer
to the work of Xu \cite{xu96}. In our context, the underlying PDE and/or the
domain can be replaced by those of a problem which is solvable by the method of
separation of variables.  The fast diagonalization method (FDM) \cite{lynch64}
is a $\bigo{p^{d+1}}$ direct factorization that breaks the problem down into a
sequence of one-dimensional subproblems.  

For the Poisson equation discretized on meshes with all cells Cartesian (all
internal angles are right angles), the vertex-star problems can be solved
directly with the FDM ~\cite{witte20}. Huismann et al.~\cite{huismann19}
introduced a remarkably fast solver with $\bigo{p^d}$ scaling on such meshes.
The linear system is statically condensed by elimination of the cell DOFs, and
the reduced system on the interface is solved with $p$-multigrid and a
restricted variant of the FDM onto the interface DOFs of a vertex-star.  Since
the statically condensed operator requires the exact inversion of the cell
matrices, their approach has no obvious extension to the unstructured,
non-Cartesian case.

The FDM can be applied as a relaxation by means of an auxiliary
problem that is separable, but this requires a tensor product grid discretization
of the patch, which is only possible when the cells are laid out in a tensor
product grid~\cite{witte20}.  When the cells are not Cartesian, the method of Witte et
al.~\cite{witte20} approximates the whole patch as a single Cartesian domain
and converges slowly even when the cells are slightly distorted.  On general
meshes, the patches may not have this structure, thus the FDM cannot be
directly applied on such patches. An example of a vertex-star patch to which
the FDM cannot be applied as a relaxation is shown in Figure \ref{fig:pentagon_patch}.

\begin{figure}[tbhp]
\captionsetup[subfloat]{margin=0pt,format=hang,singlelinecheck=true,justification=centering}
\centering
\quad\hfill
\subfloat[Vertex-star patch]{
   \includegraphics[height=0.23\textwidth,valign=c]{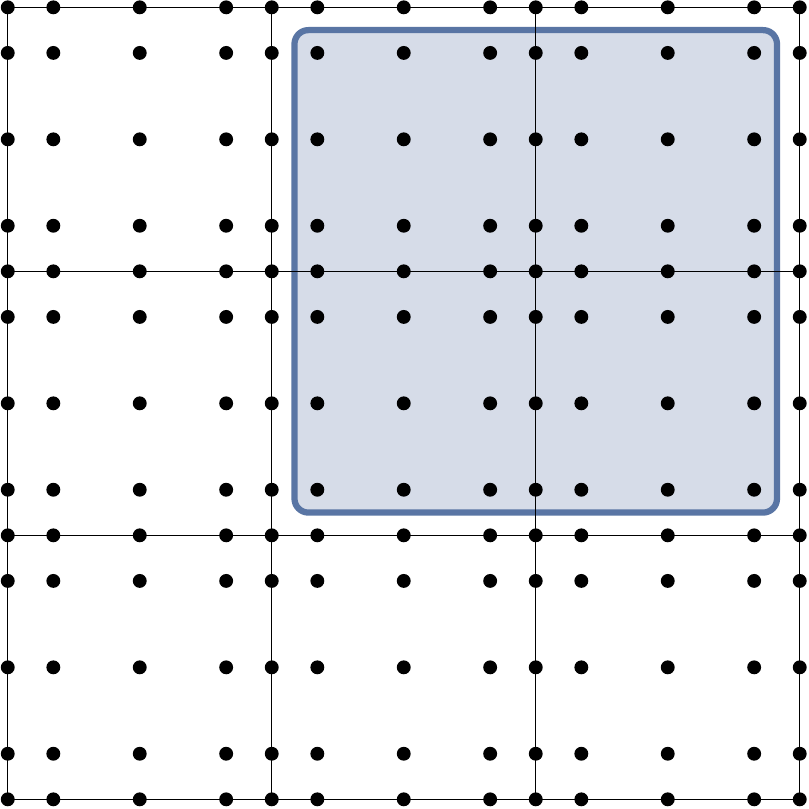}}
\hfill
\subfloat[Cell-centered patch]{
   \includegraphics[height=0.23\textwidth,valign=c]{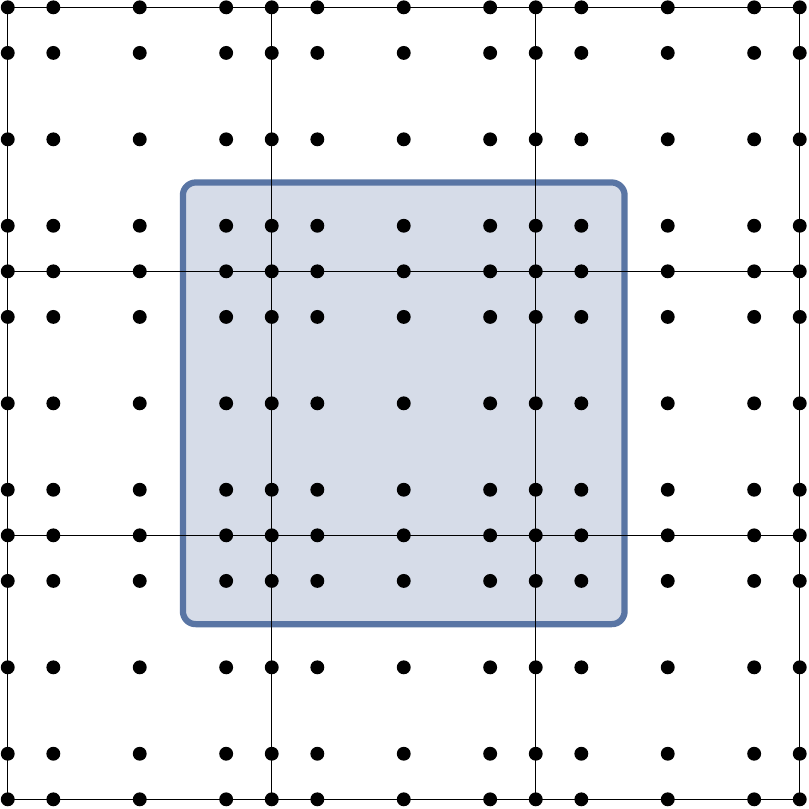}}
\hfill\quad

\caption{Subdomains for the additive Schwarz method on a regular mesh ($p=4$). In
combination with a low-order coarse grid, the
vertex-star patch gives a $p$-robust method for symmetric and coercive problems,
while the cell-centered patch does not.}
\label{fig:patches}
\end{figure}

A popular alternative in the literature has been to use cell-centered patches
with minimal overlap by including a few layers of DOFs of
the neighboring cells \cite{fischer00, remacle15, stiller17}.  This can be done
in such a way that every patch remains structured.  This kind of space
decomposition is more amenable to fast implementation, but does not give a
$p$-robust solver.  If the number of layers is fixed, then the measure of the
overlap region will decrease as $p$ is increased. Pavarino also proved that
when the overlap is not generous, the rate of convergence of the additive
Schwarz method will depend inversely on the overlap size \cite{pavarino07}. To
overcome this, Fischer and Lottes \cite{lottes05} applied a hybrid
$p$-multigrid/Schwarz method, in the context of a Poisson problem.  They
implemented several levels of $p$-multigrid to overcome the lack of
$p$-robustness of the cell-centered patches with minimal overlap.  The use of
multiple levels increases the overlap at the coarser levels with a minor impact on
the overall computational cost.  Cell-centered patches without overlap have also been
employed for non-symmetric problems~\cite{pazner18, pazner18ip, diosady19}.

\begin{figure}[tbhp]
\captionsetup[subfloat]{margin=0pt,format=hang,singlelinecheck=true,justification=centering}
\centering
\quad\hfill
\subfloat[Structured patch]{
   \includegraphics[height=0.23\textwidth,valign=c]{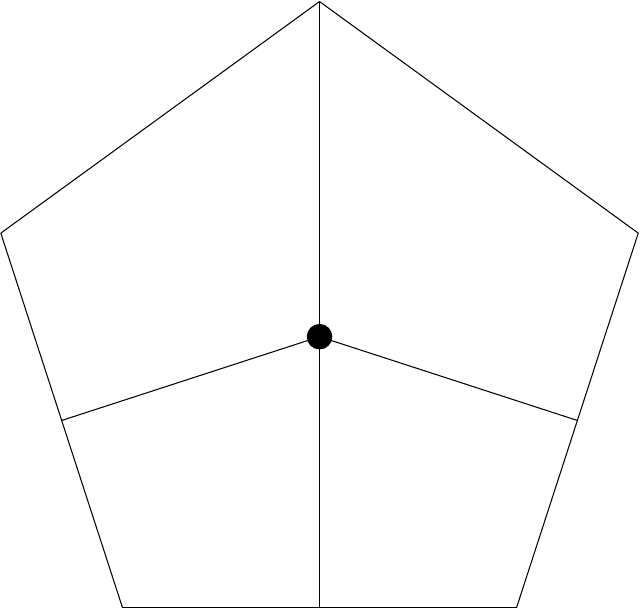}}
\hfill
\subfloat[Unstructured patch]{
   \includegraphics[height=0.23\textwidth,valign=c]{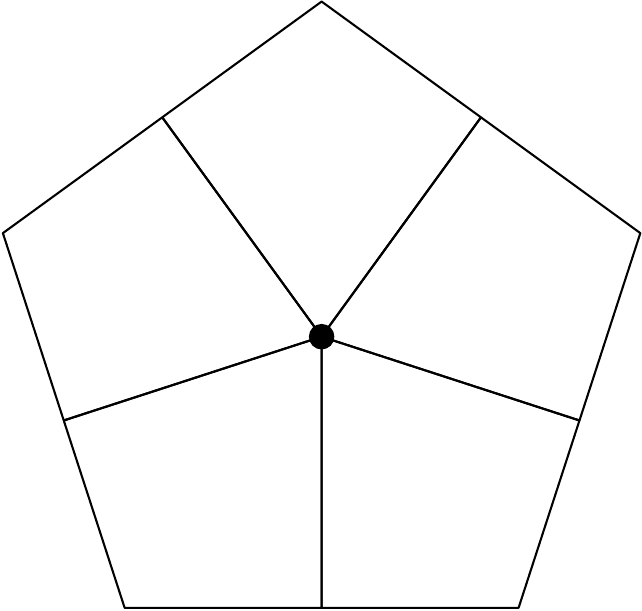}}
\hfill\quad

\caption{The FDM may be applied as a relaxation
only on vertex-star patches that are structured, 
i.e.\ where the cells are laid out in
a tensor product grid.}
\label{fig:pentagon_patch}
\end{figure}

Instead of replacing the vertex-star patches with cell-centered ones, the
alternative FEM-SEM preconditioner \cite{orszag80} rediscretizes the problem on each
vertex-star patch with $p=1$ on a GLL grid, a mesh with vertices at the DOFs of
the high-order space.  The theory behind this guarantees the spectral
equivalence between the differential operator discretized on the two spaces
\cite{canuto94}. Since low-order methods are naturally sparse, this
approach is not constrained to Cartesian cells and can deal properly with mixed
first derivatives that the FDM cannot handle. A downside of this approach is
that the Cholesky factors of the patch matrices are quite dense, limiting its
scalability to very high polynomial degree.  Computationally advantageous
approaches involve incomplete factorizations of the patch FEM-SEM matrices
\cite{pazner20}, or the use of algebraic multigrid on the global FEM-SEM
operator \cite{bello19}. 

In this work we develop a solver for vertex-star patches that scales to very
high polynomial degree. Our approach does not rely on a particular structure of
the patch. In particular, it applies to the patches shown in Figure
\ref{fig:pentagon_patch}.  The key idea is to use the FDM to numerically
construct a basis of functions on an interval that diagonalizes the interior
blocks of \emph{both} the stiffness and mass matrices in one dimension. When the
problem is assembled with respect to the associated tensor product basis, the
resulting stiffness and mass matrices are sparse, in the Cartesian case. In
particular, the total number of non-zeros is the same as that of a low-order
finite difference discretization of the Laplacian.  Moreover, fill-in in the
Cholesky factorization is only introduced for the interface DOFs, resulting in
very sparse Cholesky factors. The factorization requires $\bigo{p^{3(d-1)}}$
operations, while forward and back-substitution steps have a cost of
$\bigo{p^{2(d-1)}}$ operations that is optimal for $d \in \{2, 3\}$, in
contrast with the $\bigo{p^{3d}}$ and $\bigo{p^{2d}}$ costs of the na\"ive
approach.  A disadvantage of the approach is that the memory required scales
like $\bigo{p^{2(d-1)}}$, instead of the optimal $\bigo{p^d}$ required for
storing the solution. In the non-Cartesian case, we approximate the form with
one that is separable in the reference problem.  Robustness with respect to $h$
and $p$ should follow from the spectral equivalence between the forms, and
numerical experiments indicate that the approach is effective when the cells
are moderately deformed.

We demonstrate the effectiveness of our approach by applying it to a $H(\div)
\times L^2$ conforming discretization of a mixed formulation of incompressible
linear elasticity.  We present a sequence of problems of increasing complexity
building up to this.  In Section \ref{sec:poisson} we present the standard
$hp$-FEM formulation for the Poisson problem and construct a solver based on a
sparse discretization of an auxiliary locally separable PDE that employs the
numerically computed FDM basis. In Section \ref{sec:elasticity} we consider the
application of our solver to linear elasticity.  In the primal formulation,
although our approach can be applied to patch problems for the individual
components of displacement, we explain why it cannot be applied to the coupled
vector-valued problem, which is necessary for parameter-robustness in the
incompressible regime. We therefore consider a mixed formulation instead.
Developing a $p$-robust solver requires both a $p$-robust preconditioner and a
$p$-robust discretization, and for the latter we choose a $H(\div) \times L^2$
conforming approach. We then extend the method to symmetric interior-penalty
discontinuous Galerkin discretizations, required for the displacement block of
the mixed problem. We apply our relaxation to the displacement block of the
incompressible elasticity system in conjunction with block-preconditioned
Krylov methods. We end with conclusions in Section \ref{sec:conclusion}.

\section{Sparse Poisson solver\label{sec:poisson}}

We will first describe a solver for the Poisson equation on Cartesian cells,
which will subsequently be extended as a preconditioner for more general
symmetric coercive problems on non-Cartesian cells.

\subsection{Continuous Galerkin formulation}
We start from the standard weak formulation of the Poisson equation. Consider a
bounded domain $\Omega \subset \reals^d$, $d\in\{1,2,3\}$, and let $\Gamma_D \subseteq
\partial\Omega$ be the part of the boundary where the Dirichlet boundary
condition $\left.u\right|_{\Gamma_D} = u_0$ is prescribed. The problem is to
find $u-u_0$ in $V \coloneqq H^1_0(\Omega) = \{v\in H^1(\Omega),
\left.v\right|_{\Gamma_D}=0\}$ such that
\begin{equation}
a(v, u) = L(v) \quad \forall \, v \in V,
\end{equation}
where
\begin{equation}
a(v,u) \coloneqq \int_{\Omega} \nabla v \cdot \nabla u ~\md\bx, \quad
L(v) \coloneqq \int_{\Omega} v f ~\md\bx.
\end{equation}

The standard FEM discretization employs a mesh $\mesh = \{K\}$ of $\Omega$. In
this work we consider quadrilateral and hexahedral cells, so that each cell $K$
can be mapped with a diffeomorphism $F_K: \hat{K} \to K$ from the reference
hypercube $\hat{K} = \hat{\mathcal{I}}^d$, where $\hat{\mathcal{I}} = [-1, 1]$
is the reference interval. The approximate solution $u_h \in V_h$ is sought in
the space of piecewise continuous tensor product polynomials on each cell,
i.e.\ $V_h \coloneqq \mathrm{Q}_p (\Omega) \subset V$. We first define the
space of shape functions on $\hat{K}$
\begin{equation}
\mathrm{Q}_p(\hat{K}) \coloneqq \bigotimes_{j=1}^d \mathrm{P}_p(\hat{\mathcal{I}}),
\quad \mathrm{P}_p(\hat{\mathcal{I}}) \coloneqq \spn\left\{\hat{x}^j, 0\le j\le p\right\},
\end{equation}
and via composition with $F_K^{-1}$, we define
\begin{equation}
\mathrm{Q}_p(\Omega) \coloneqq \left\{v \in H^1_0(\Omega): \forall \, K\in \mesh \ \exists \ \hat{v}\in \mathrm{Q}_p(\hat{K}) \text{ s.t. }
\left. v\right|_{K} = \hat{v}\circ F_K^{-1}\right\}.
\end{equation}

Once we fix a basis $\{\phi_j\}_{j=1}^n$ for $V_h$, the approximate solution
is expanded as $u_h = \sum_{j=1}^n u_j \phi_j$. The resulting $n\times n$  
system of linear equations is
\begin{equation}
A \uu = \uf,
\end{equation}
where $[A]_{ij} = a(\phi_i, \phi_j)$ is the stiffness matrix, $\uu = (u_1,
\ldots, u_n)^\top$ is the vector of coefficients, and $\uf = (L(\phi_1),
\ldots, L(\phi_n))^\top$ is the load vector.

We recall the standard construction of the basis $\{\phi_j\}$~\cite{karniadakis13}. 
The basis is defined in terms of shape functions
$\{\hat{\phi}_j\}$ on $\hat{K}$. Given shape functions
$\{\hat{\phi}^\mathrm{1D}_j\}_{j=0}^p$ for $\mathrm{P}_p(\hat{\mathcal{I}})$, a
tensor product basis $\{\hat{\phi}_j\}$ for $\mathrm{Q}_p(\hat{K})$ can be
constructed as
\begin{equation}
\hat{\phi}_j(\hat{\bx}) = \prod_{k=1}^d \hat{\phi}^\mathrm{1D}_{j_k}(\hat{x}_k),
\end{equation}
where we have expanded $j = (j_1, \ldots, j_d)\in [0,p]^d$ as a multi-index.

The interval shape functions are decomposed into interface and interior
modes.  The interface modes have non-zero support on either endpoint of
$\hat{\mathcal{I}}$, while the interior modes vanish at the boundary of
$\hat{\mathcal{I}}$.  In multiple dimensions, the shape functions decompose
into interior, facet, edge, and vertex modes,  depending on how many 1D
interface functions are multiplied together.  To generate a $C^0$ basis, we
simply match the shape of individual interface modes.  Hence, $A$ will be block
sparse, since $[A]_{ij} = 0$ when $i$ and $j$ correspond to interior modes
supported on different cells.

For the interval shape functions, one standard choice is the
set of Lagrange polynomials on the Gau{\ss}--Lobatto--Legendre (GLL) nodes
$\{\hat{\xi}_i\}_{i=0}^p \subset [-1,1]$.  These nodes are the roots of
$(1-\hat{\xi}^2) P'_{p}(\hat{\xi})$, where $P_k(\hat{\xi})$ is the Legendre 
polynomial of degree $k$. The Lagrange polynomials $\{\ell_j(\hat{x})\}$
satisfy $\ell_j(\hat{\xi}_i) = \delta_{ij}$ by construction, 
\begin{equation}
\ell_j(\hat{x}) 
= \prod_{k=0, k\neq j}^p \frac{ \hat{x} - \hat{\xi}_k}
{\hat{\xi}_j - \hat{\xi}_k}, 
\quad j = 0, \ldots, p.
\end{equation}
Another useful basis is formed by the hierarchical Lobatto shape functions $\{l_j\}$,
which are constructed by augmenting the so-called bubble functions (integrated
Legendre polynomials) with linear Lagrange functions,
\begin{equation}
l_j(\hat{x}) =
\begin{cases}
(1 - \hat{x})/2 & \mbox{for } j = 0, \\
(1 + \hat{x})/2 & \mbox{for } j = p, \\
\int_{-1}^{\hat{x}} P_j(z) ~\md z & \mbox{for } j = 1,\ldots, p-1.
\end{cases}
\end{equation}
These two choices of shape functions are plotted in Figure \ref{fig:basis}.

The assembly of the stiffness matrix $A$ is described as follows.  On each cell
$K\in \mesh$ we define the cell stiffness matrix
$A^K\in\reals^{(p+1)^d\times (p+1)^d}$ in terms of the basis functions
$\{\phi^K_j\}$ that are supported on $K$, which are obtained from the reference
shape functions $\{\hat{\phi}_j\}$ via $\phi^K_j = \hat{\phi}_j \circ
F_K^{-1}$. Then, the cell stiffness matrices are
\begin{equation}
[A^K]_{ij}
= \int_K \nabla \phi^K_i \cdot \nabla \phi^K_j ~\md\bx. 
\end{equation}
The global stiffness matrix is then assembled via direct stiffness summation:
\begin{equation}
A = \sum_{K\in\mesh} R_K^\top A^K R_K,
\end{equation}
where $R_K\in \reals^{(p+1)^d\times n}$ is the Boolean restriction matrix from
the global DOFs to those local to cell $K$.

\begin{figure}[tbhp]
\centering
\quad\hfill
\subfloat[GLL, $\ell_j(\hat{x})$]{
   \includegraphics[height=0.24\textwidth,valign=c]{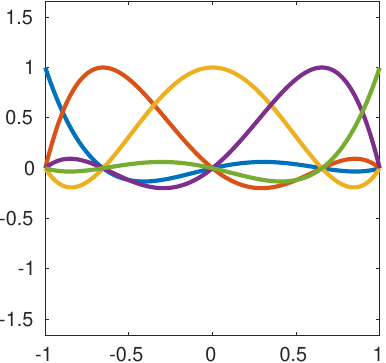}}
\hfill
\subfloat[GLL, $\hat{B}$]{
   \includegraphics[height=0.24\textwidth,valign=c]{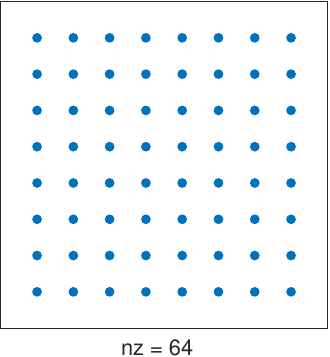}}
\hfill
\subfloat[GLL, $\hat{A}$]{
   \includegraphics[height=0.24\textwidth,valign=c]{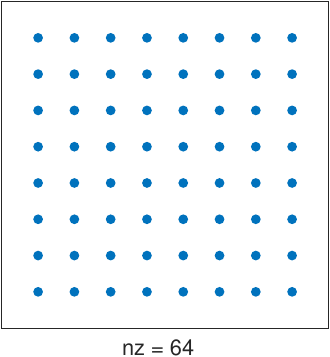}}
\hfill\quad

\quad\hfill
\subfloat[Hierarchical, $l_j(\hat{x})$]{
   \includegraphics[height=0.24\textwidth,valign=c]{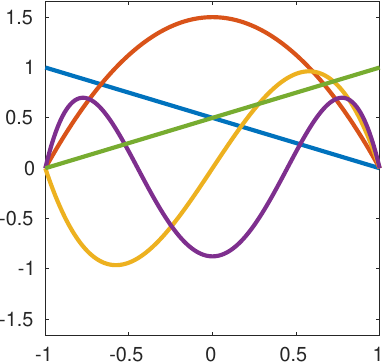}}
\hfill
\subfloat[Hierarchical, $\hat{B}$]{
   \includegraphics[height=0.24\textwidth,valign=c]{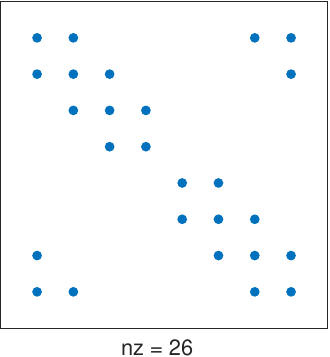}}
\hfill
\subfloat[Hierarchical, $\hat{A}$]{
   \includegraphics[height=0.24\textwidth,valign=c]{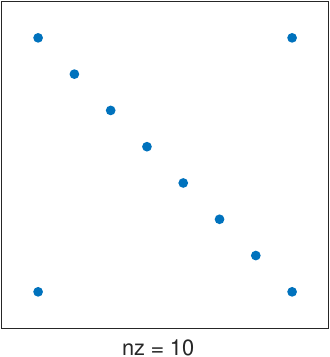}}
\hfill\quad

\caption{
Plots of the interval shape functions ($p=4$) and non-zero structure
of the mass and stiffness matrices on the reference interval ($p=7$).}
\label{fig:basis}
\end{figure}

\subsection{Tensor product structure on Cartesian cells}

If $d=1$ and $F_K$ is an affine mapping, then the cell stiffness
matrices are
\begin{equation}
[A^K]_{ij}
= \frac{1}{L^K} \int_{\hat{\mathcal{I}}} \hat{\phi}'_i\hat{\phi}'_j ~\md \hat{x}
= \frac{1}{L^K} [\hat{A}]_{ij}.
\end{equation}
Here $\hat{A} \in \reals^{(p+1)\times(p+1)}$ is the stiffness matrix on the reference interval,
and $L^K =\abs{K}/|\hat{\mathcal{I}}|$, where $\abs{K}$
denotes the measure of the cell $K$.

For $d=\{2,3\}$, we will first consider the case where $\Omega$ can be
tessellated with a mesh $\mesh$ consisting of Cartesian cells, i.e.\ the cells
are rectangular quadrilaterals or hexahedra (all internal angles are right
angles). In this idealized setting, the cell stiffness matrices are separable
in the reference coordinates
\begin{equation}
\label{eq:kron}
A^K = \begin{cases}
\mu^K_1 \hat{B}\kron \hat{A} +
\mu^K_2 \hat{A}\kron \hat{B} & \mbox{if~}d=2,\\
\mu^K_1 \hat{B}\kron \hat{B}\kron \hat{A} +
\mu^K_2 \hat{B}\kron \hat{A}\kron \hat{B} +
\mu^K_3 \hat{A}\kron \hat{B}\kron \hat{B} & \mbox{if~}d=3,
\end{cases}
\end{equation}
where
\begin{equation}
[\hat{B}]_{ij} = \int_{\hat{\mathcal{I}}} \hat{\phi}_i ~\hat{\phi}_j ~\md \hat{x}, 
\end{equation}
is the mass matrix on the reference interval, $\mu^K_j = (L^K_j)^{-2}
\prod_{i=1}^d L^K_i$, and $L^K_j$ is the length of $K$ along the $j$-th axis
divided by $|\hat{\mathcal{I}}|$. The symbol $\kron$ denotes the
Kronecker product, which for matrices $A\in \reals^{m\times n}, B \in
\reals^{r\times s}$, is defined as the block matrix
\begin{equation}
A\kron B = \begin{bmatrix}
a_{11} B & \cdots & a_{1n} B\\
\vdots &\ddots & \vdots\\
a_{m1} B & \cdots & a_{mn} B
\end{bmatrix} \in \reals^{rm\times sn}.
\end{equation}
It follows that if $A$ and $B$ are sparse, then $A\kron B$ is also sparse.

For the GLL basis $\{\ell_j\}$, both $\hat{A}$ and $\hat{B}$ are
dense, but these are sparse in the hierarchical basis $\{l_j\}$.  This
is illustrated in Figure \ref{fig:basis}. On affine cells, the hierarchical basis
yields a sparse stiffness matrix.  The bubble functions $\{l_j\}_{j=1}^{p-1}$,
satisfy $l'_j(\hat{x}) = P_j(\hat{x})$, and due to the orthogonality of the
Legendre polynomials, the interior block of $\hat{A}$ is diagonal. The only
off-diagonal non-zeros in $\hat{A}$ are due to the coupling between the
interface modes $l_0$, $l_p$.  Nevertheless, in order for this sparsity to
propagate to higher dimensions, we would additionally wish that $\hat{B}$ is
also as sparse as possible.  This is not quite the case for $\{l_j\}$, as
$\hat{B}$ has two interior blocks with tri-diagonal structure, in the 
even-odd decomposition. Therefore, on a
typical row, $A$ will have the structure of the 5-point stencil for $d=2$ and
that of a 13-point stencil for $d=3$.

\subsection{The fast diagonalization method}
Linear systems involving structured matrices such as that defined in
\eqref{eq:kron} can be solved efficiently using the fast diagonalization method
(FDM) \cite{lynch64}.  This method reduces the computation into a sequence of
eigenvalue problems on the interval in a similar fashion as the method of
separation of variables.  It requires a separable PDE and a tensor product
basis; therefore it can only be applied on meshes or mesh patches with tensor
product structure.

To illustrate the FDM, we consider solving a problem on the interior of a
Cartesian cell, $A^K_{II} \uu^K_I = \ur^K_I$, where $A^K_{II} = R^K_I A
R^K_I{}^\top$ and $R^K_I \in \reals^{(p-1)^d \times n}$ is the Boolean
restriction matrix onto the interior DOFs of $K$. 
We may first solve the generalized eigenvalue problem on
the interior of the reference interval
\begin{equation} \label{eq:eig}
\hat{A}_{II} \hat{S}_{II} = \hat{B}_{II} \hat{S}_{II} \hat{\Lambda}_{II},
\end{equation}
in conjunction with the normalization condition $\hat{S}_{II}^\top \hat{B}_{II}
\hat{S}_{II} = \eye$.  Here $\hat{A}_{II}, \hat{B}_{II} \in \reals^{(p-1)
\times (p-1)}$ are the interior blocks of $\hat{A}$ and $\hat{B}$, respectively,
$\hat{\Lambda}_{II} \in \reals^{(p-1)\times (p-1)}$ is the diagonal matrix of
eigenvalues, and $\hat{S}_{II} \in \reals^{(p-1)\times (p-1)}$ is the matrix of
eigenvectors. The generalized eigenproblem \eqref{eq:eig} may be
equivalently rewritten as
\begin{equation} \label{eq:eigenproblem}
\hat{S}_{II}^\top \hat{A}_{II} \hat{S}_{II}^{} = \hat{\Lambda}_{II}, \quad 
\hat{S}_{II}^\top \hat{B}_{II} \hat{S}_{II}^{} = \eye.
\end{equation}
The corresponding continuous problem is to find $s_j(\hat{x})$ and $\lambda_j$, $j = 1, \dots, p-1,$
such that
\begin{equation}
   \int_{\hat{\mathcal{I}}} s'_i(\hat{x}) s'_j(\hat{x}) ~\md\hat{x} = \lambda_j \delta_{ij}, \quad 
   \int_{\hat{\mathcal{I}}} s_i(\hat{x}) s_j(\hat{x}) ~\md\hat{x} = \delta_{ij}, \quad s_j(\pm 1) = 0,
\end{equation}
with solution $s_j(\hat{x}) =\sin(j\pi (1+\hat{x})/2)$, $\lambda_j = j^2\pi^2/4$.
Thus, when $\hat{A}$ and $\hat{B}$ are discretized with the GLL basis, 
$\hat{S}_{II}$ approximates a discrete sine transform on the interior GLL nodes.

If $A^K$ is given by \eqref{eq:kron}, then its inverse has the following diagonal
factorization
\begin{equation}
\left(A^K_{II}\right)^{-1} = 
\left\{\bigotimes_{k=1}^d \hat{S}_{II}^{} \right\}
\left(\Lambda^K_{II}\right)^{-1}
\left\{\bigotimes_{k=1}^d \hat{S}_{II}^\top \right\},
\end{equation}
where
\begin{equation}
\Lambda^K_{II} = 
\begin{cases}
\mu^K_1 \eye \kron \hat{\Lambda}_{II} +
\mu^K_2 \hat{\Lambda}_{II} \kron \eye  & \mbox{if~} d=2, \\
\mu^K_1 \eye \kron \eye \kron \hat{\Lambda}_{II} +
\mu^K_2 \eye \kron \hat{\Lambda}_{II} \kron \eye +
\mu^K_3 \hat{\Lambda}_{II} \kron \eye \kron \eye  & \mbox{if~} d=3. \\
\end{cases}
\end{equation}
Therefore, the solution of a system $A^K_{II} \uu^K_I = \ur^K_I$ can be
obtained with $\bigo{p^{d+1}}$ computational work.

The main limitation of this approach is that it does not generalize to terms
that contain first derivatives, ruling out the possible extension to advection
problems.  Mixed first derivative terms are also very common in symmetric
coercive problems, for instance, when the cells have non-orthotropic
deformations, or for vector-valued operators that mix first derivatives of
distinct vector components, such as $\nabla (\nabla\cdot\bu$).

\subsection{Sparse FDM discretization}

Our key idea is to construct a new finite element basis on the interval,
inspired by the FDM, which yields a sparse stiffness matrix.  To construct this
new basis, we solve \eqref{eq:eigenproblem} with $\hat{A}$ and $\hat{B}$ in 
the GLL basis. Then, we interpolate the eigenvectors of the FDM with polynomials
$\{\hat{s}_j\}_{j=1}^{p-1} \subset \mathrm{P}_p(\hat{\mathcal{I}})$ that
satisfy $\hat{s}_{j}(\pm 1)=0$, $\hat{s}_j(\hat{\xi}_i) = [\hat{S}_{II}]_{ij}$
for $i,j\in [1,p-1]$.  The unisolvent dual basis of our proposed FDM
discretization of $\mathrm{P}_p(\hat{\mathcal{I}})$ consists of point
evaluation at the vertices and integral moments against the orthogonal interior
shape functions $\hat{s}_{j}$, $j=1,\ldots, p-1$.  This construction ensures
that $\hat{A}_{II}$ and $\hat{B}_{II}$ become diagonal under this basis.  As an
additional consequence, the 1D interface shape functions are also orthogonal to
the interior ones, but not to each other.  Hence $\hat{B}$ becomes as sparse as
possible in this new FDM basis. Figure \ref{fig:fdmshape} shows the shape
functions and the non-zero structure of $\hat{A}$ and $\hat{B}$ for the FDM
basis.

\begin{figure}[tbhp]
\centering
\quad\hfill
\subfloat[FDM, $\hat{s}_j(\hat{x})$]{
   \includegraphics[height=0.24\textwidth,valign=c]{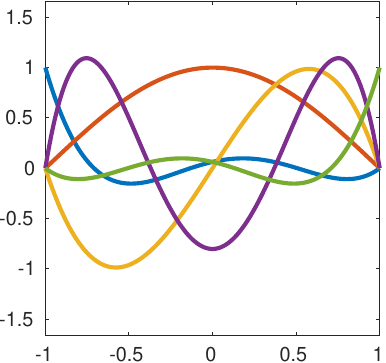}}
\hfill
\subfloat[FDM, $\hat{B}$]{
   \includegraphics[height=0.24\textwidth,valign=c]{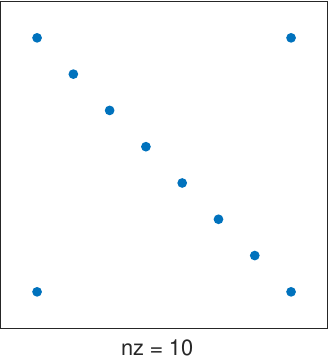}}
\hfill
\subfloat[FDM, $\hat{A}$]{
   \includegraphics[height=0.24\textwidth,valign=c]{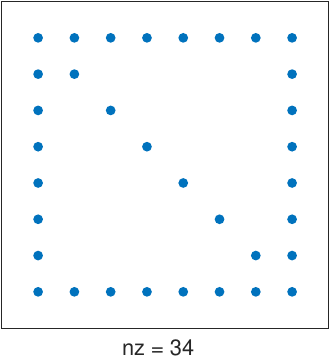}}
\hfill\quad

\caption{
Plots of the interval FDM shape functions ($p=4$) and non-zero structure of the
mass and stiffness matrices in the FDM basis on the reference interval ($p=7$).}
\label{fig:fdmshape}
\end{figure}

The stiffness matrix for a Cartesian cell $A^K$ in this FDM basis is also given
by \eqref{eq:kron}, and has structured sparsity, 
since $A\kron B$ is sparse when $A$ and $B$ are sparse. 
Consequently, the global matrix $A$ is also sparse.  We can thus apply sparse direct factorization methods or
other preconditioners such as the additive Schwarz method.

For problems that are separable in Cartesian coordinates, we obtain a solver that
is reminiscent of the FDM.  The multiplication times the Kronecker product of
the matrices of eigenvectors is incorporated in the definition of the FDM
shape functions.  If for some reason the problem needs to be solved with any
other basis, one can transform the linear system back and forth to the FDM
basis via interpolation and restriction at each application of the method.  We
also replace the diagonal matrix of the traditional FDM with a sparse matrix,
which can be assembled even on 3D meshes with a very large number of cells and
for very high degree.  This approach removes the requirement of a global tensor
product grid while exploiting the separability of the PDE and the local
structure of $\mathrm{Q}_p(\Omega)$, at the expense of replacing the diagonal
factor by a sparse matrix. 

\subsection{Hybrid $p$-multigrid/Schwarz method}

The solver of Pavarino is fully additive, across both the coarse grid and the
vertex-star patches. In our work we consider a small variation of this, with the
solver multiplicative between the two levels while remaining additive among the
vertex-star patches. This improves the convergence at essentially no cost. The
method can be interpreted as a hybrid multiplicative two-level
$\mathrm{V}(1,1)$-cycle with the additive Schwarz method~\cite{dryja87} with
vertex-star patches as the fine grid relaxation and the lowest-order
discretization on the same mesh as the coarse space. The sparse matrix for the coarse space may be
assembled and factorized, or other preconditioners such as geometric or algebraic multigrid
may be applied instead. The vertex-star patch $V_j$ includes the degrees of
freedom associated with vertex $\bv_j$ of $\mesh$ and all cells, facets, and
edges adjacent to $\bv_j$ (the topological entities in the \emph{star} of the
vertex, a standard concept in algebraic topology~\cite[\S 2]{munkres1984}).

We may write the multigrid relaxation as
\begin{equation}
P_\mathrm{ASM}^{-1} = \sum_{j=1}^J \bar{R}_j^\top A_j^{-1} \bar{R}_j^{},
\end{equation}
with $\bar{R}_j$ the Boolean restriction matrix onto $V_j$, and
$A_j = \bar{R}_j^{} A \bar{R}_j^\top$
are the sparse patch matrices for which we may explicitly compute a Cholesky
decomposition. The relaxation is scaled by the damping coefficient 
\begin{equation}
\omega = 2 \left[
(1+\alpha)\lambda_\mathrm{max} +
(1-\alpha)\lambda_\mathrm{min}
\right]^{-1}, 
\end{equation}
where $\lambda_\mathrm{\min}, \lambda_\mathrm{\max}$ are the extremal
eigenvalues of $P_\mathrm{ASM}^{-1}A$ estimated via the CG-Lanczos procedure
\cite{lanczos50}, and $\alpha = 0.25$ is chosen to tackle the high frequency error,
also ensuring that the error iteration matrix $\eye-\omega 
P_\mathrm{ASM}^{-1}A$ is contractive.

To illustrate the direct solver on the Cartesian vertex-star patch shown in
Figure \ref{fig:patches}a, we show in Figure \ref{fig:spy} the non-zero
structure for the patch matrix $A_j$ and its Cholesky factor. The sparsity
pattern of the global matrix $A$ connects the interior DOFs to their
projections onto the facets, hence a typical interior row of $A$ will have
$2d+1$ non-zeros. For the patch matrix $A_j$, an interior row will only have
$d+1$ non-zeros, as the patch only includes one facet per dimension on each
cell.  Moreover, the total number of non-zeros of $A_j$ is the same as that of
a low-order finite difference or finite element discretization with the 5-point
or 7-point stencil on the same grid.

\begin{figure}[tbhp]
\centering
\subfloat[$A_j$, $d=2$]{
   \includegraphics[height=0.24\textwidth,valign=c]{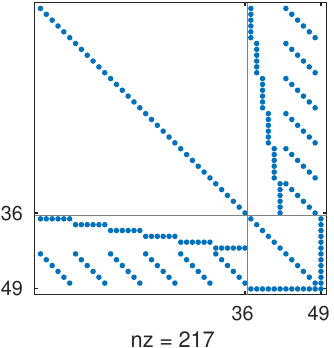}}
\hfill
\subfloat[$\mathrm{chol}(A_j)$, $d=2$]{
   \includegraphics[height=0.24\textwidth,valign=c]{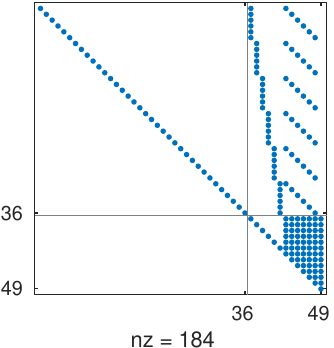}}
\hfill
\subfloat[$A_j$, $d=3$]{
   \includegraphics[height=0.24\textwidth,valign=c]{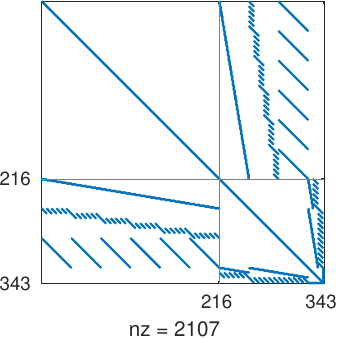}}
\hfill
\subfloat[$\mathrm{chol}(A_j)$, $d=3$]{
   \includegraphics[height=0.24\textwidth,valign=c]{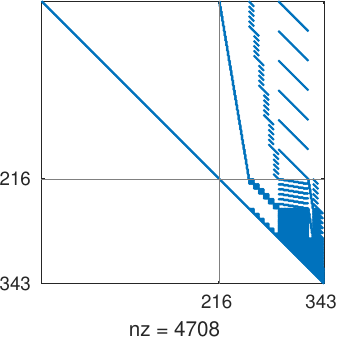}}

\caption{
Non-zero structure of the stiffness matrix in the FDM basis $ A_j = L_j
L_j^\top$ and its upper Cholesky factor $L_j^\top$ for the Poisson problem on a
Cartesian vertex-star patch with $p=4$. With the nested dissection reordering,
the factor matrix has minimal fill-in, occurring only on the bottom-right
block.}
\label{fig:spy}
\end{figure}

\subsection{Computational complexity\label{sec:complexity}}

Here we discuss the computational cost of
the solution of the patch problem using the Cholesky factorization. 
Once the factorization has been computed, it may be applied in $\bigo{p^{2(d-1)}}$ cost,
which is optimal for $d\in \{2,3\}$.  Unfortunately, the factorization phase is
suboptimal, requiring $\bigo{p^{3(d-1)}}$ operations to compute. The memory
required to store the Cholesky factor is $\bigo{p^{2(d-1)}}$, which for $d=3$
is one factor of $p$ higher than that required to store the solution.

Consider a stiffness matrix $A$ discretized with the FDM basis on any
mesh with all cells Cartesian.  The number of floating point operations (flops)
needed to solve a linear system using a sparse Cholesky factorization
$A=LL^\top$ is roughly four times the number of non-zero entries in $L$
\cite{davis08}.  To maximize the sparsity in $L$, it is crucial to reorder the
DOFs, such that interior DOFs are followed by the interface DOFs.  This ensures
that the fill-in is introduced only on the bottom-left block. To analyze the
cost of factorization, we first introduce the block $LDL^\top$ decomposition
\begin{equation}
\label{eq:LDLT}
A = 
\begin{bmatrix}
A_{II} & A_{I\Gamma}\\
A_{\Gamma I} & A_{\Gamma\Gamma}
\end{bmatrix}
=
\begin{bmatrix}
\eye & 0\\
A_{\Gamma I} A_{II}^{-1}& \eye
\end{bmatrix}
\begin{bmatrix}
A_{II} & 0\\
0 & S_{\Gamma}
\end{bmatrix}
\begin{bmatrix}
\eye &  A_{II}^{-1}A_{I\Gamma}\\
0 & \eye
\end{bmatrix},
\end{equation}
where $S_{\Gamma} = A_{\Gamma\Gamma} -
A_{\Gamma I} A_{II}^{-1} A_{I \Gamma}$ is
the interface Schur complement.
By construction, the top-left block $A_{II}$ is diagonal with
positive entries, with Cholesky factor $A_{II}^{1/2}$.
If we decompose $A_{II}$ and $S_{\Gamma}$ in the second
matrix on the RHS of \eqref{eq:LDLT} into
their Cholesky factors, and distribute each factor onto the other two matrices, we
obtain the Cholesky decomposition of $A$:
\begin{equation}
A = L L^\top = 
\begin{bmatrix}
A_{II}^{1/2} & 0\\
A_{\Gamma I} A_{II}^{-1/2}& L_\Gamma
\end{bmatrix}
\begin{bmatrix}
A_{II}^{1/2} &  A_{II}^{-1/2}A_{I\Gamma}\\
0 & L_\Gamma^\top
\end{bmatrix},
\end{equation}
where the Schur complement is factorized as $S_{\Gamma} =
L_\Gamma^{} L_\Gamma^\top$.  Since $A_{II}$ is diagonal, the off-diagonal block
$A_{\Gamma I} A_{II}^{-1/2}$ will preserve the non-zero structure of $A_{\Gamma
I}$, and similarly for its transpose. Thus, fill-in is only introduced
on the interface block through $L_\Gamma$.

An ordering strategy that minimizes fill-in consists of applying nested
dissection \cite{george73} on the adjacency graph that connects topological
entities. Each node in this graph represents a cell, face, edge or vertex.
The ordering of the entities is then used to permute the
corresponding blocks in $A$.

Assuming for the worst case that $L_\Gamma^\top$ is dense, the memory required
to store the Cholesky factor is $\bigo{p^{2(d-1)}}$. This represents a
significant increase from the traditional FDM, which is kept at the optimal
$\bigo{p^{d}}$. However, the DOF ordering does lead to some structured sparsity
in $L_\Gamma^\top$, as can be seen in Figure \ref{fig:spy}.
Nevertheless, we still observe dense $\bigo{p^{d-1}}\times\bigo{p^{d-1}}$ blocks.
The fact that $L_\Gamma^\top$ contains these dense blocks indicates
that $\bigo{p^{3(d-1)}}$ operations are required in the factorization phase.
However,
the forward and back-substitution steps have a computational
cost of $\bigo{p^{2(d-1)}}$ operations, which is optimal for $d\le 3$.

Compared to the FEM-SEM approach, our approach with the FDM basis has a sparser
Cholesky factorization.  In Figure \ref{fig:memory} we present the ratio of the
number of non-zeros in the Cholesky factors of our approach and the FEM-SEM
preconditioner ordered with nested dissection for $d \in \{2, 3\}$. The fact
that the ratio is always below 1 confirms that our approach is sparser, with a
substantial gain at higher degrees.  Practical FEM-SEM solvers use AMG
\cite{bello19} or patchwise multigrid with ILU smoothers \cite{pazner20} to
avoid the cost of the Cholesky factorizations of the patch matrices.

\begin{figure}
\centering
\resizebox{0.3\textwidth}{!}{
\begin{tikzpicture}
    \begin{semilogxaxis}[
         max space between ticks=30,
         ymin = 0, ymax = 1,
         ylabel=Relative non-zeros,
         xlabel=$p$,
         xtick = {3, 7, 15, 31},
         minor xtick = {4,5,6, 9,11,13, 19,23,27},
         log x ticks with fixed point,
         x tick label style={/pgf/number format/1000 sep=\,},
         legend cell align={left},
         legend style={at={(0.9,0.9)},anchor=north east},
         mark size=3pt, line width=0.75pt]
   \addplot[color=paired6,mark=*,mark options={fill=paired5}] table[x=p, y=2D] {figures/memory.txt};
   \addplot[color=paired2,mark=triangle*,mark options={fill=paired1}] table[x=p, y=3D] {figures/memory.txt};
   \legend{$d=2$\\ $d=3$\\}
   \end{semilogxaxis}
\end{tikzpicture}
}
\caption{
Relative number of non-zero entries in the Cholesky factors of the stiffness
matrix with the FDM approach and the FEM-SEM preconditioner, for a Cartesian vertex-star patch of $2^d$ cells. 
The FDM approach is sparser, with substantial gains at higher degrees.
}
\label{fig:memory}
\end{figure}

\subsection{Extension to non-Cartesian cells}

For arbitrarily deformed cells, the local stiffness matrices $A^K$ cannot be
expressed in terms of tensor products of $\hat{A}$ and $\hat{B}$ as in
\eqref{eq:kron}, and $A^K$ is not sparse in the FDM basis.  The
preconditioning techniques found in \cite{couzydeville95, fischer00, witte20}
introduce an auxiliary Cartesian domain to construct a separable problem for
which the FDM is a direct solver.  The method described by Fischer
\cite{fischer00} constructs a preconditioner by replacing $K$ with its nearest
rectangular approximation, whose dimensions are computed as the mean separation
between the mapped GLL nodes from opposite facets of $K$.  Witte et al.\
\cite{witte20} obtain the lengths from the average arclength of opposite sides
of $K$, but it is not clear how this extends to the 3D case. To the best of our
knowledge, no theory underpins these choices.

Our approach to construct the separable surrogate is based on the theory of
equivalent operator preconditioning \cite{axelsson09}. We work with the
bilinear form $a(\cdot, \cdot)$ in terms of the reference coordinates.  We
discard the mixed derivative terms that prevent separability, and we
replace the coefficients with piecewise constants in the reference
coordinates\footnote{Recall that piecewise constant coefficients in the
physical coordinates will not yield piecewise constant coefficients in  the
reference coordinates.}. We will prove that this choice yields a spectrally
equivalent operator.

The bilinear form $a(\cdot, \cdot)$ can be expressed as a sum of cell
contributions $a_K(\cdot,\cdot)$ where integration and differentiation are with
respect to $\hat{\bx}$.  The measure $\md\bx$ is replaced by
$\abs{\det(\Jac_K)} \md\hat{\bx}$ and the gradient is computed via the chain
rule, since the arguments of the form become functions of $\hat{\bx}$ after
being composed with $F_K$. Hence,
\begin{equation}
a(v, u) 
= \sum_{K\in\mesh} a_K(v, u) 
= \sum_{K\in\mesh} \int_{\hat{K}} 
\hat{\nabla} v\circ F_K \cdot \hat{G}^K  \hat{\nabla}
u\circ F_K
~\md\hat{\bx},
\end{equation}
where $\hat{\nabla}$ is the gradient with respect to $\hat{\bx}$,
and $\hat{G}^K : \hat{K} \to \reals^{d\times d}$ 
is the inverse metric of the coordinate transformation
weighted by the Jacobian determinant,
\begin{equation}
\hat{G}^K =
\abs{\det(\Jac_K)}
\Jac_K^{-1}
\Jac_K^{-\top}.
\end{equation}
This tensor encapsulates all of the geometry-dependent information in the form;
it is spatially dependent for generally-deformed elements, and constant in the
case of affine transformations.  For a separable geometry, $\hat{G}^K$ is
diagonal, and thus for a Cartesian cell it is diagonal and constant.  To
construct an auxiliary problem that is separable by the FDM in the reference
coordinates, we replace $\hat{G}^K$ in $a_K(\cdot, \cdot)$ with a constant
diagonal approximation $\diag(\mu^K_j)$.  Each $\mu^K_j$ is given by the
cell-wise average of the diagonal entry $\hat{G}^K_{jj}$,
\begin{equation}
\mu^K_j \coloneqq \frac{1}{|\hat{K}|}
\int_{\hat{K}} \hat{G}^K_{jj} ~\md\hat{\bx},
\end{equation}
where summation over the index $j$ is not implied.
As the approximation is local to each cell, it is still possible to assemble
a sparse stiffness matrix discretizing the auxiliary problem
on meshes where cells are not structured in a tensor product grid.

We now establish the spectral equivalence between the original bilinear form
and the auxiliary separable one.
\begin{theorem}
\label{thm:spectral}
Let $\hat{\mu}_K \coloneqq \diag(\mu^K_j)$ be the constant diagonal approximation
of $\hat{G}^K$, and define the auxiliary bilinear form
\begin{equation}
\tilde{a}(v, u) \coloneqq
 \sum_{K\in\mesh} \tilde{a}_K(v, u) 
\coloneqq \sum_{K\in\mesh} \int_{\hat{K}} 
\hat{\nabla} v\circ F_K \cdot \hat{\mu}_K  \hat{\nabla}
u\circ F_K
~\md\hat{\bx}.
\end{equation}
Then, there exist $p$-independent constants $c$, $C > 0$ that depend on $\mesh$
through $\hat{G}^K$ such that
\begin{equation}
c\leq \frac{a(v,v)}{\tilde{a}(v,v)} \leq C \quad \forall \, v\in V \setminus \{0\}.
\end{equation}
\end{theorem}

\begin{proof}
Let $c_K, C_K$ be lower and upper bounds for the spectrum of the diagonally scaled metric,
so that $\sigma(\hat{\mu}_K^{-1/2} \hat{G}^K \hat{\mu}_K^{-1/2}) \in [c_K, C_K]$ for
all $\hat{\bx}\in\hat{K}$. We claim that
\begin{equation}
c_K \leq \frac{a_K(v,v)}{\tilde{a}_K(v,v)} \leq C_K 
\quad \forall \, v\in \{v \in V: \left.v\right|_K \neq 0\}.
\end{equation}
This result is obtained by first rewriting $a_K(v,v)$ with
$\hat{\mu}_K^{1/2}
\hat{\mu}_K^{-1/2}\hat{G}^K\hat{\mu}_K^{-1/2}
\hat{\mu}_K^{1/2}$
instead of $\hat{G}^K$, and then replacing $\hat{\mu}^{-1/2}_K \hat{G}^K
\hat{\mu}^{-1/2}_K$ with $c_K \eye$ or $C_K \eye$ to find the lower or upper
bounds, respectively.
It then follows that
\begin{equation}
 c\coloneqq \min_{K\in\mesh} c_K
\leq \frac{a(v,v)}{\tilde{a}(v,v)} 
\leq \max_{K\in \mesh}C_K \eqqcolon C
\quad \forall \, v\in V \setminus \{0\}.
\end{equation}
\end{proof}

Let $\tilde{A}$ be the stiffness matrix associated with the auxiliary form
$\tilde{a}(\cdot, \cdot)$. By Theorem \ref{thm:spectral}, the condition number
$\kappa(\tilde{A}^{-1}A)$ is bounded by $C/c$ independently of $p$. Numerical
experiments also indicate that $\kappa(\tilde{A}^{-1}A)$ is independent of $h$
under uniform refinements.  Now consider a preconditioner $P$ where the
auxiliary form $\tilde{a}(\cdot, \cdot)$ is used additively in both the coarse
solve and the vertex-star patches. In this case, Theorem 1 of \cite{pavarino93}
guarantees that $\kappa(P^{-1}\tilde{A})$ is bounded from above independently
of $h$ and $p$. Hence we may conclude that $\kappa(P^{-1}A) \leq
\kappa(P^{-1}\tilde{A}) \kappa(\tilde{A}^{-1}A)$ is bounded independently of
$h$ and $p$. In practice, we expect that using multiplicative coarse grid
correction with the original form $a(\cdot, \cdot)$ can only improve the
preconditioner.

To gain useful insight, we consider the case where $d=2$ and $F_K$ is an 
affine transformation, that is when $K$ is a parallelogram. 
Without loss of generality, suppose that one of the sides of $K$ has length 
$2L_1$ and is aligned with the first reference coordinate axis, and the other 
side of length $2L_2$ is at an angle $\theta$ with respect to the same axis.
The Jacobian of the coordinate transformation is
\begin{equation}
\Jac_K =
\begin{bmatrix}
L_1 & L_2\cos{\theta}\\
0 & L_2\sin{\theta}
\end{bmatrix},
\end{equation}
to which corresponds the Jacobian-weighted inverse metric
\begin{equation}
\hat{G}^K =
\frac{1}{L_1 L_2\abs{\sin{\theta}}}
\begin{bmatrix}
L_2^2 & -L_1 L_2\cos{\theta}\\
-L_1 L_2\cos{\theta} & L_1^2
\end{bmatrix}.
\end{equation}
Since $\hat{G}^K$ is constant, $\hat{\mu}_K$ is simply the diagonal part of
$\hat{G}^K$. The spectrum of the diagonally scaled metric will be independent
of $L_1$ and $L_2$, but still depend on $\theta$,
\begin{equation} \label{eqn:spectralboundstheta}
\sigma\left(\hat{\mu}_K^{-1/2} \hat{G}^K \hat{\mu}_K^{-1/2}\right)
= \left[1-\abs{\cos{\theta}}, 1+\abs{\cos{\theta}}\right].
\end{equation}
This spectrum is desirable because it is centered at 1 and bounded above for
all $\theta$.  If we follow the geometric approaches of
\cite{fischer00,witte20}, we would have to choose a rectangle with side lengths
$2L_1$ and $2L_2$ as the auxiliary domain for the Poisson problem. Then, the
previous bounds \eqref{eqn:spectralboundstheta} would become scaled by
$\abs{\sin{\theta}}^{-1}$. In this case, the spectrum is unbounded from above
in the limit $\theta\to 0$.

\subsection{Numerical experiments}

We provide an implementation of the $\mathrm{P}_p$ element with the FDM shape
functions on the interval in the FIAT \cite{kirby04} package.  The extension to
quadrilaterals and hexahedra is achieved by taking tensor products of the one-dimensional
element with FInAT \cite{homolya17}.  Code for the sum-factorized
evaluation of the residual is automatically generated by Firedrake
\cite{Rathgeber2016, homolya18}, implementing
a Gau\ss--Lobatto quadrature rule with $3(p+1)/2$ points along each direction.
The sparse preconditioner discretizing the auxiliary form is implemented in
$\texttt{firedrake.FDMPC}$ using PETSc \cite{petsc-user-ref}. The Cholesky
factorization of the patch matrices is computed using CHOLMOD~\cite{davis08}.
Most of our computations were performed using an Intel Xeon CPU E5-4627 v2
$@$ 3.30GHz with 32 cores and 67.6 GB of RAM storage.

The hybrid $p$-multigrid/Schwarz solver employing the FDM/sparse relaxation is
illustrated in Figure \ref{fig:solver_pcg}.  
To achieve scalability with respect to the mesh parameter $h$, on the $p$-coarse
problem we employ geometric multigrid with damped point-Jacobi relaxation and a
Cholesky factorization on the coarsest level using MUMPS~\cite{mumps01}.
We test the effectiveness of this approach on a hierarchy of meshes obtained by
$l \ge 0$ uniform refinements of the base meshes shown in Figure \ref{fig:basemesh}.

\begin{figure}[tbhp]
\footnotesize
\centering
\begin{tikzpicture}[%
 every node/.style={draw=black, thick, anchor=west},
grow via three points={one child at (-0.0,-0.7) and
	two children at (0.0,-0.7) and (0.0,-1.4)},
edge from parent path={(\tikzparentnode.210) |- (\tikzchildnode.west)}]
\node {Krylov solver: PCG}
child {node {Hybrid $p$-multigrid/Schwarz V-cycle}
   child {node {Relaxation: FDM/sparse}
   }
   child {node {$p$-coarse: geometric multigrid}
      child {node {Relaxation: point-Jacobi}}
      child {node {$h$-coarse: Cholesky}}
   }
};
\end{tikzpicture}
\caption{Solver diagram for the Poisson problem.}
\label{fig:solver_pcg}
\end{figure}

We present results for the Poisson equation in $\Omega=(0,1)^d$ discretized on
the three hierarchies of Cartesian, unstructured, and structured but deformed
(Kershaw) \cite{kershaw81} meshes. 
The coordinate field of the Kershaw mesh is in $[\mathrm{Q}_3(\Omega)]^d \cap C^1(\Omega)$, with
a cell aspect ratio of $\varepsilon_y = \varepsilon_z = 0.3$
near the corners of the domain. 
We impose homogeneous Dirichlet BCs on $\Gamma_D = \partial\Omega$ and a
constant forcing $f = 1$. 
In Table \ref{tab:poisson-iter} we present PCG iteration counts required to
reduce the Euclidean norm of the residual by a factor of $10^8$ starting from a
zero initial guess.  In Table \ref{tab:poisson-cond} we show the condition
number $\kappa(P^{-1}A)$ estimated by CG-Lanczos. 
The results show almost complete $p$- and $h$-robustness in the Cartesian case,
and very slow growth of iteration counts in the unstructured case. 
Given the lack of shape regularity, the Kershaw mesh is significantly more
challenging; even with exact patch solvers, we do not expect $h$- or
$p$-robustness.

\begin{figure}[tbhp]
\centering
\subfloat[Cartesian]{
   \includegraphics[height=0.23\textwidth,valign=c]{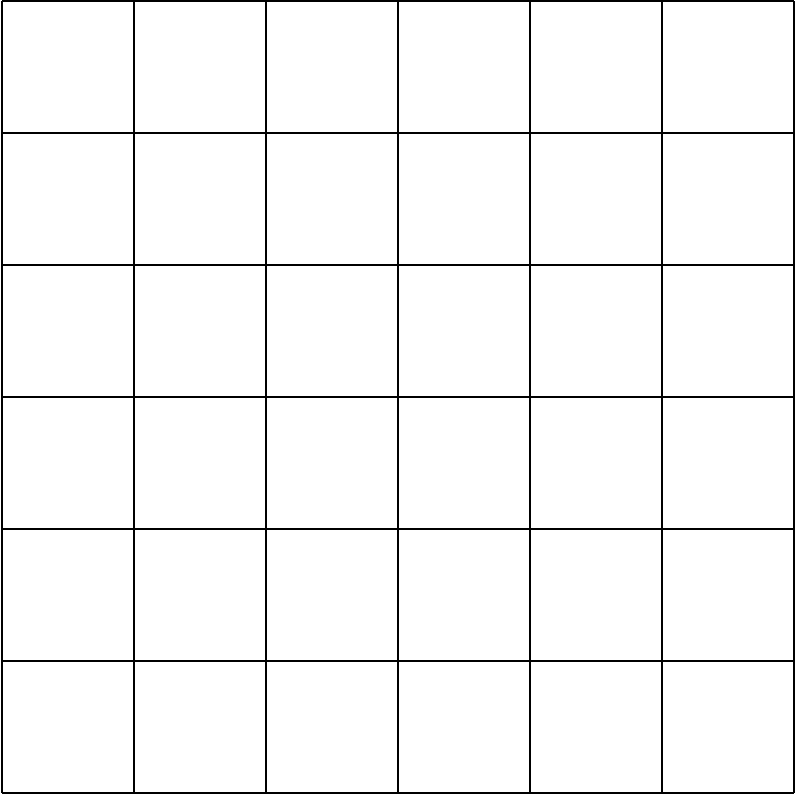}}
\hfill
\subfloat[Unstructured]{
   \includegraphics[height=0.23\textwidth,valign=c]{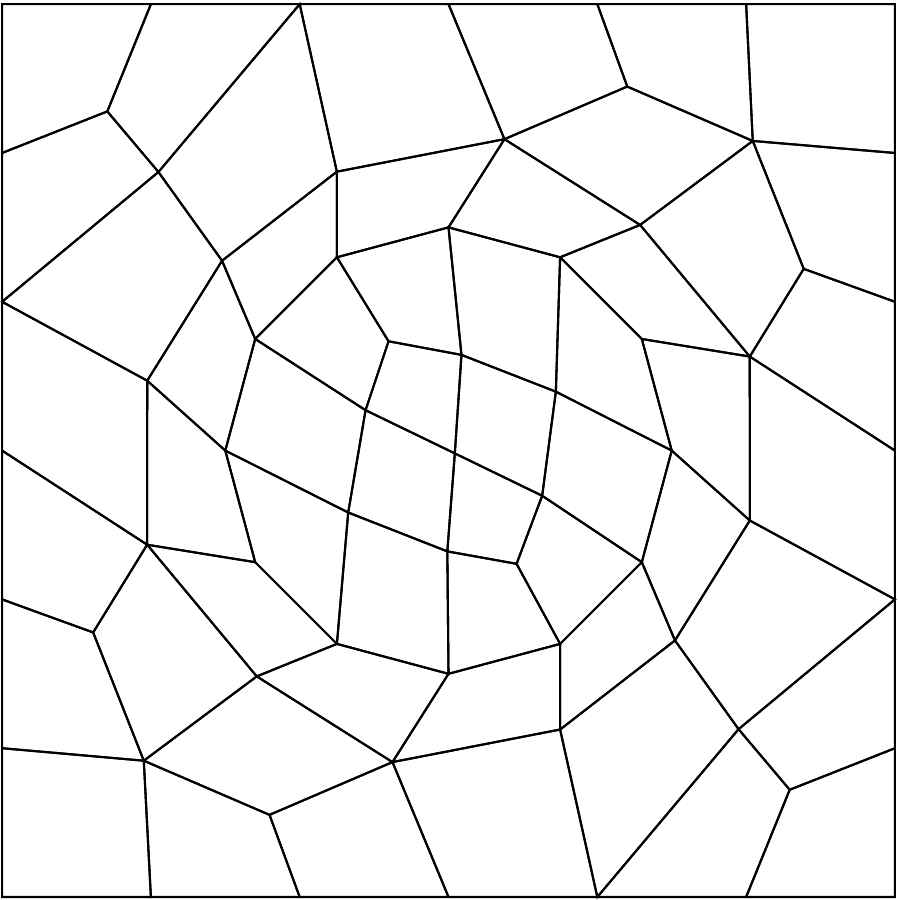}}
\hfill
\subfloat[Kershaw, $d=2$]{
   \includegraphics[height=0.23\textwidth,valign=c]{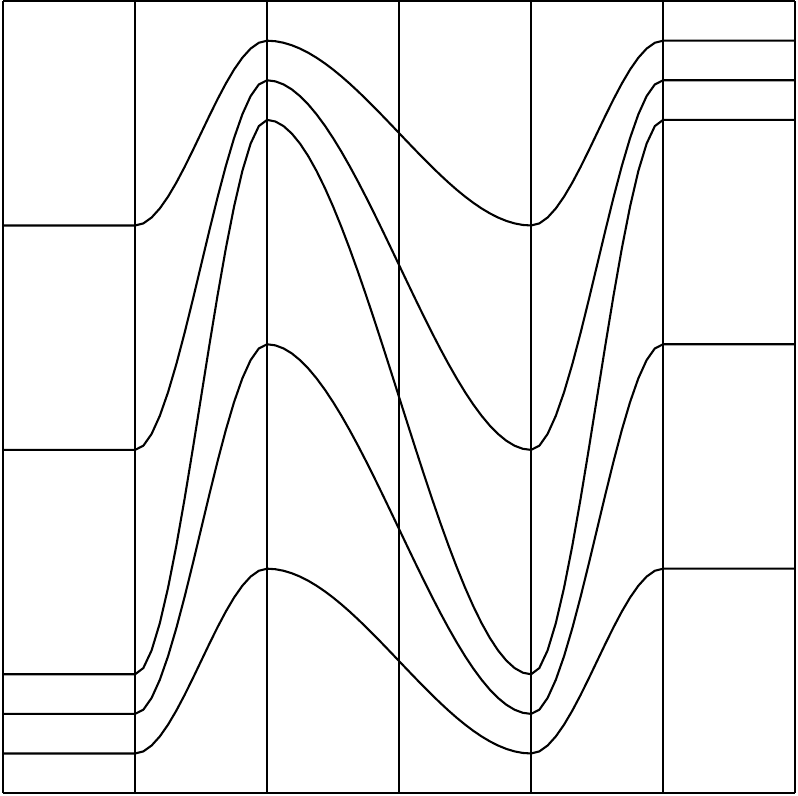}}
\hfill
\subfloat[Kershaw, $d=3$]{
   \includegraphics[height=0.23\textwidth,valign=c]{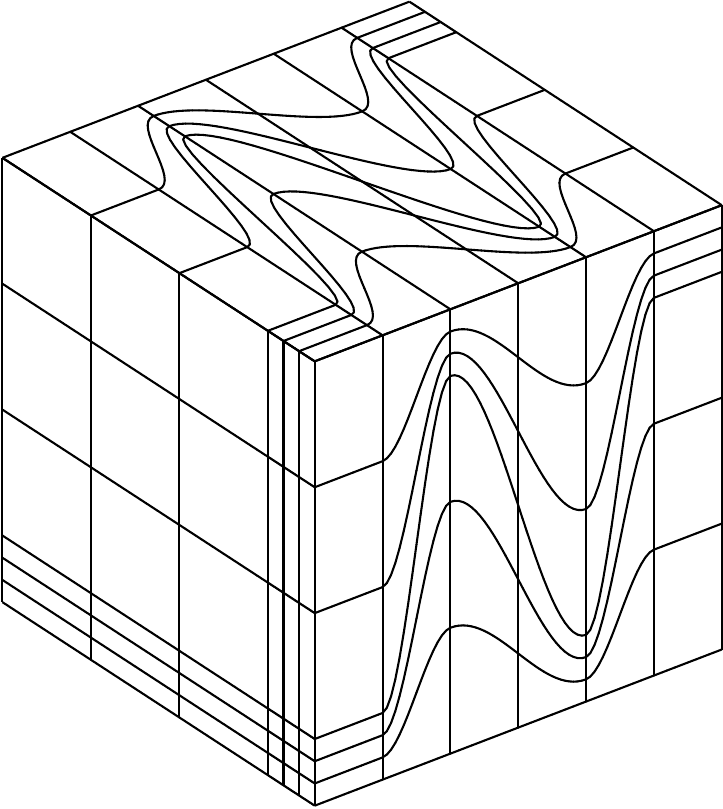}}

\caption{Base meshes for the Poisson problem. The Cartesian and unstructured base meshes used
for $d=3$ are the extrusion with six layers of the two-dimensional meshes shown here.}
\label{fig:basemesh}
\end{figure}

\begin{table}[tbhp]
{
\footnotesize
\caption{
PCG iteration counts for 
the hybrid $p$-multigrid/Schwarz solver with the FDM/sparse relaxation.
The patch problems are solved exactly on the Cartesian mesh.}
\label{tab:poisson-iter}
\begin{center}
\rowcolors{4}{}{gray!25}
\begin{tabular}{rr|*{3}{r}|*{3}{r}|*{3}{r}}
\toprule
& & \multicolumn{3}{c|}{Cartesian} 
  & \multicolumn{3}{c|}{Unstructured}
  & \multicolumn{3}{c}{Kershaw}
\\
$d$ & $p\setminus l$ 
& \multicolumn{1}{c}{0} & \multicolumn{1}{c}{1} & \multicolumn{1}{c|}{2} 
& \multicolumn{1}{c}{0} & \multicolumn{1}{c}{1} & \multicolumn{1}{c|}{2} 
& \multicolumn{1}{c}{0} & \multicolumn{1}{c}{1} & \multicolumn{1}{c}{2} 
\\
\midrule
2  & 3 & 7 & 8 & 9 & 12 & 13 & 14 & 27 & 35 & 54\\
   & 7 & 8 & 8 & 9 & 16 & 16 & 17 & 44 & 56 & 78\\
   &15 & 8 & 9 & 9 & 19 & 19 & 19 & 58 & 69 & 90\\
   &31 & 8 & 9 & 9 & 21 & 20 & 21 & 67 & 80 & 97\\
\midrule
3  & 3 & 12 & 12 & 12 & 17 & 17 & 18 & 54  & 66  & 102\\
   & 7 & 12 & 12 & 12 & 22 & 21 & 21 & 98  & 106 & 158\\
   &15 & 12 & 13 &    & 25 & 24 &    & 131 & 132 &    \\
\bottomrule
\end{tabular}
\end{center}
}
\end{table}

\begin{table}[tbhp]
{
\footnotesize
\caption{
Estimated condition numbers for the preconditioned operator $\kappa(P^{-1}A)$
using the hybrid $p$-multigrid/Schwarz solver with the FDM/sparse relaxation.
}
\label{tab:poisson-cond}
\begin{center}
\rowcolors{4}{}{gray!25}
\begin{tabular}{rr|*{3}{r}|*{3}{r}|*{3}{r}}
\toprule
& & \multicolumn{3}{c|}{Cartesian} 
  & \multicolumn{3}{c|}{Unstructured}
  & \multicolumn{3}{c}{Kershaw}
\\
$d$ & $p\setminus l$ 
& \multicolumn{1}{c}{0} & \multicolumn{1}{c}{1} & \multicolumn{1}{c|}{2} 
& \multicolumn{1}{c}{0} & \multicolumn{1}{c}{1} & \multicolumn{1}{c|}{2} 
& \multicolumn{1}{c}{0} & \multicolumn{1}{c}{1} & \multicolumn{1}{c}{2} 
\\
\midrule
2  & 3 & 1.44 & 1.49 & 1.50 & 2.14 & 2.37 & 2.81 & 9.34 & 15.6 & 34.5\\
   & 7 & 1.48 & 1.48 & 1.50 & 3.23 & 3.27 & 3.79 & 19.6 & 30.3 & 57.6\\
   &15 & 1.51 & 1.51 & 1.52 & 4.06 & 3.78 & 4.13 & 30.5 & 45.8 & 69.0\\
   &31 & 1.54 & 1.52 & 1.52 & 4.45 & 4.06 & 4.36 & 40.4 & 57.1 & 73.3\\
\midrule
3  & 3 & 2.87 & 2.49 & 2.45 & 4.16 & 4.21 & 4.55 & 34.8 & 46.1 & 117\\
   & 7 & 2.79 & 2.70 & 2.67 & 5.88 & 5.54 & 5.47 & 100 & 110   & 266\\
   &15 & 2.83 & 2.79 &      & 7.12 & 6.44 &      & 165 & 151   &    \\
\bottomrule
\end{tabular}
\end{center}
}
\end{table}

To assess the computational performance of our approach, we solve the
three-dimensional Poisson equation on a Cartesian mesh with $3\times 3\times 3$
cells with a single core of an Intel Core i7-10875H CPU $@$ 2.30GHz.
We plot in Figure \ref{fig:runtime} the runtimes, flop counts, and achieved
arithmetic performance for the Cholesky factorization of the patch matrices,
the solution of the patch problems using this factorization 
(per application of the relaxation), and the matrix-free
evaluation of the residual (per Krylov iteration, excluding the application of the global to local map) as functions of $p$. 
The dotted lines are to indicate powers of $2p-1$, which is the number of DOFs
along each side of typical vertex-star patch not intersecting the mesh boundary.

Despite the $\bigo{p^{3(d-1)}}$ computational cost of the Cholesky
factorization, these results show $\bigo{p^{2(d-1)}}$ scaling for runtime up to
$p=15$.  This speedup can be explained mainly by data locality.  The sparse
Cholesky factorization is obtained by recursively applying the block
$LDL^\top$ decomposition up to the point where the Schur complement is
sufficiently dense.  The computation of this Schur complement via dense
matrix-matrix multiplication (BLAS-3 $\texttt{dgemm}$) dominates the
computational cost.  As $p$ is increased, the utilization of arithmetic units
increases in proportion to the dimension of the Schur complement, which
explains the $\bigo{p^{d-1}}$ scaling of the achieved arithmetic performance.
As the arithmetic capabilities become saturated for $p>15$, the
$\bigo{p^{3(d-1)}}$ scaling in the factorization runtime should become
apparent.

Most of the time in the relaxation step is spent in accessing the factor matrix
from memory, given the $\bigo{p^{2(d-1)}}$ sub-optimal storage per patch.  The
relaxation is therefore limited by memory bandwidth and not arithmetically
intense, which explains the poor arithmetic performance.
This is in contrast to the sum-factorized residual evaluation, 
which has a $\bigo{p^d}$ memory footprint and
presents better arithmetic utilization \cite{kronbichler18}.
Nevertheless, the results indicate that the runtime for the
solution of the patch problems with the sparse Cholesky factorization remains
very close to that of the matrix-free residual evaluation for moderate $p$,
being slightly faster for $p\leq 7$, mainly due to lower operation count.

\begin{figure}[tbhp]
\centering

\subfloat[Runtime]{\resizebox{0.33\textwidth}{!}{
\begin{tikzpicture}
    \begin{loglogaxis}[
         max space between ticks=20,
         ylabel=Runtime (s),
         xlabel=$p$,
         xtick = {3, 7, 15, 31},
         minor xtick = {4,5,6, 9,11,13, 19,23,27},
         log x ticks with fixed point,
         x tick label style={/pgf/number format/1000 sep=\,},
         legend cell align={left},
         legend style={at={(0.05,0.9)},anchor=north west},
         mark size=3pt, line width=0.75pt]
   \addplot[color=paired8,mark=*,mark options={fill=paired7}] table[x=p, y expr=\thisrow{factorsym_time}] {figures/flops3.txt};
   \addplot[color=paired6,mark=*,mark options={fill=paired5}] table[x=p, y expr=\thisrow{factornum_time}] {figures/flops3.txt};
   \addplot[color=paired4,mark=square*,mark options={fill=paired3}] table[x=p, y expr=\thisrow{solve_time}/\thisrow{iter}] {figures/flops3.txt};
   \addplot[color=paired2,mark=triangle*,mark options={fill=paired1}] table[x=p, y expr=\thisrow{resid_time}/\thisrow{iter}] {figures/flops3.txt};
   \addplot[color=black,mark=none,dashed] table[x=p,y expr=4E-4*((2*\thisrow{p}-1)/5)^4] {figures/flops3.txt}
      coordinate [pos=0.75] (A) 
      coordinate [pos=0.90] (B)
   ;
   \draw (A) -| (B)  
      node [pos=0.75, anchor=west] {4}
      node [pos=0.25, anchor=north] {1} 
   ;
   \addplot[color=black,mark=none,dashed] table[x=p,y expr=4E-4*((2*\thisrow{p}-1)/5)^6] {figures/flops3.txt}
      coordinate [pos=0.75] (C) 
      coordinate [pos=0.90] (D)
   ;
   \draw (C) |- (D)  
      node [pos=0.25, anchor=east] {6} 
      node [pos=0.75, anchor=south] {1}
   ;
   \legend{factor (sym)\\ factor (num)\\ patch solve\\ residual\\}
   \end{loglogaxis}
\end{tikzpicture}
}}
\hfill
\subfloat[Flop count]{\resizebox{0.33\textwidth}{!}{
\begin{tikzpicture}
    \begin{loglogaxis}[
         max space between ticks=20,
         ylabel=Gflop,
         xlabel=$p$,
         xtick = {3, 7, 15, 31},
         minor xtick = {4,5,6, 9,11,13, 19,23,27},
         log x ticks with fixed point,
         x tick label style={/pgf/number format/1000 sep=\,},
         legend cell align={left},
         legend style={at={(0.1,0.9)},anchor=north west},
         mark size=3pt, line width=0.75pt]
   \addplot[color=paired6,mark=*,mark options={fill=paired5}] table[x=p, y expr=1E-9*\thisrow{factornum_flops}] {figures/flops3.txt};
   \addplot[color=paired4,mark=square*,mark options={fill=paired3}] table[x=p, y expr=1E-9*\thisrow{solve_flops}/\thisrow{iter}] {figures/flops3.txt};
   \addplot[color=paired2,mark=triangle*,mark options={fill=paired1}] table[x=p, y expr=1E-9*\thisrow{resid_flops}/\thisrow{iter}] {figures/flops3.txt};
   \addplot[color=black,mark=none,dashed] table[x=p,y expr=1E-3*((2*\thisrow{p}-1)/5)^4] {figures/flops3.txt}
      coordinate [pos=0.80] (A) 
      coordinate [pos=0.95] (B)
   ;
   \draw (A) |- (B)  
      node [pos=0.25, anchor=east] {4} 
      node [pos=0.75, anchor=south] {1}
   ;
   \addplot[color=black,mark=none,dashed] table[x=p,y expr=1E-3*((2*\thisrow{p}-1)/5)^6] {figures/flops3.txt}
      coordinate [pos=0.80] (C) 
      coordinate [pos=0.95] (D)
   ;
   \draw (C) |- (D)  
      node [pos=0.25, anchor=east] {6} 
      node [pos=0.75, anchor=south] {1}
   ;
   \end{loglogaxis}
\end{tikzpicture}
}}
\hfill
\subfloat[Performance]{\resizebox{0.33\textwidth}{!}{
\begin{tikzpicture}
    \begin{loglogaxis}[
         ylabel=Gflop/s,
         xlabel=$p$,
         xtick = {3, 7, 15, 31},
         minor xtick = {4,5,6, 9,11,13, 19,23,27},
         log x ticks with fixed point,
         x tick label style={/pgf/number format/1000 sep=\,},
         legend cell align={left},
         legend style={at={(0.1,0.9)},anchor=north west},
         mark size=3pt, line width=0.75pt]
   \addplot[color=paired6,mark=*,mark options={fill=paired5}] table[x=p, y expr=1E-9*\thisrow{factornum_flops}/\thisrow{factornum_time}] {figures/flops3.txt};
   \addplot[color=paired4,mark=square*,mark options={fill=paired3}] table[x=p, y expr=1E-9*\thisrow{solve_flops}/\thisrow{solve_time}] {figures/flops3.txt};
   \addplot[color=paired2,mark=triangle*,mark options={fill=paired1}] table[x=p, y expr=1E-9*\thisrow{resid_flops}/\thisrow{resid_time}] {figures/flops3.txt};
   \addplot[color=black,mark=none,dashed] table[x=p,y expr=4.25E-1*((2*\thisrow{p}-1)/5)^2] {figures/flops3.txt}
      coordinate [pos=0.35] (A) 
      coordinate [pos=0.50] (B)
   ;
   \draw (A) -| (B)  
      node [pos=0.75, anchor=west] {2}
      node [pos=0.25, anchor=north] {1} 
   ;
   \end{loglogaxis}
\end{tikzpicture}
}}

\caption{
Runtimes, flop counts, and achieved arithmetic performance for the 
Cholesky factorization (symbolic and numeric), solution of
the patch problems, and residual evaluation
for a Cartesian mesh with $3\times 3 \times 3$ cells on a single CPU core. 
We observe that the factorization runtime scales better
than expected, close to the optimal $\bigo{p^{d+1}}$ complexity.
}
\label{fig:runtime}
\end{figure}

\section{Application to linear elasticity problems\label{sec:elasticity}}

\subsection{Primal formulation of linear elasticity\label{sec:primal}}

We now consider how these ideas may be applied in the more complex setting of a
nonseparable, vector-valued PDE.  The equations of linear elasticity describe
the displacement $\bu: \Omega \to \reals^d$ of a solid body with a reference
configuration $\Omega \subset \reals^d$. The primal formulation is to find
$\bu-\bu_0 \in V \coloneqq [H^1_0(\Omega)]^d$ such that
\begin{equation}
a(\bv, \bu) = L(\bv) \quad \forall \, \bv \in V,
\end{equation} 
where
\begin{equation}
a(\bv, \bu) = \int_{\Omega}
2\mu \varepsilon(\bv) : \varepsilon(\bu) + \lambda \nabla\cdot\bv \nabla\cdot\bu ~\md\bx, \quad
L(\bv) = \int_{\Omega} \bv\cdot \bB ~\md\bx. 
\end{equation}
Here, we assume that the material is homogeneous and isotropic, and can
thus be described by the Lam\'e parameters $\mu, \lambda>0$; $\varepsilon(\bu)
= (\nabla \bu + \nabla \bu ^\top)/2$ is the linearized strain tensor; $\bu_0$
is Dirichlet data prescribed on $\Gamma_D\subseteq \partial\Omega$; and $\bB \in
[L^2(\Omega)]^d$ is a body force. The Poisson ratio $\nu = \lambda/(2\mu +
2\lambda)$ measures the compressibility of the material.  In the incompressible
limit $\lambda \to \infty$ (i.e.~$\nu \to 1/2$), the problem becomes
ill-conditioned, as $a(\cdot, \cdot)$ becomes insensitive to divergence-free
perturbations in the arguments.

Consider the partitioning of the stiffness
matrix $A$ into blocks that act on each displacement component,
\begin{equation}
A = \begin{bmatrix}
A_{11} & \cdots & A_{1d} \\
\vdots & \ddots & \vdots \\
A_{d1} & \cdots & A_{dd}
\end{bmatrix}.
\end{equation}
The diagonal block $A_{jj}$ discretizes the bilinear form
\begin{equation}
\int_\Omega \mu \nabla v_j \cdot \nabla u_j + (\mu + \lambda) \frac{\partial v_j}{\partial x_j} \frac{\partial u_j}{\partial x_j} ~\md\bx,
\end{equation}
where summation is not implied, 
and $u_j$ and $v_j$ are components of $\bu$ and $\bv$, respectively.
The off-diagonal blocks $A_{ij}$, $i\neq j$, discretize
\begin{equation}
\int_\Omega \mu \frac{\partial v_i}{\partial x_j} \frac{\partial u_j}{\partial x_i} 
      + \lambda \frac{\partial v_i}{\partial x_i} \frac{\partial u_j}{\partial x_j} ~\md\bx.
\end{equation}
The diagonal blocks can be diagonalized by the FDM on the interior of a
Cartesian cell when the reference axes are aligned with the physical coordinates.
The same statement does not hold true for the off-diagonal blocks, as they
couple together different displacement components. This is because they
discretize products of different first derivatives on the different components
and hence are not separable.

The separate displacement components (SDC) preconditioner \cite{blaheta94,
gustafsson98} is defined as the block diagonal matrix $A_\mathrm{SDC} =
\diag(A_{11}, \ldots, A_{dd})$. In other words, this approach is also described
as block-Jacobi in the displacement components.  
The SDC preconditioner discretized with the FDM basis is sparse
for Cartesian cells aligned with the coordinate axes.
On arbitrary cells, for each separate component, we obtain
an auxiliary form that is separable in the reference coordinates by selecting
constant diagonal coefficients $\hat{\mu}_K$. 

It is shown in \cite{blaheta94} that for a homogeneous isotropic material with
principal axes parallel to the axes of the reference coordinate system, the
condition number of the preconditioned matrix will depend on the Poisson ratio:
\begin{equation} \label{eq:SDC}
\kappa(A_\mathrm{SDC}^{-1}A) \leq \frac{d-1}{\gamma} \frac{1-\nu}{1-2\nu},
\end{equation}
where $\gamma$ is the constant appearing in Korn's inequality,
\begin{equation}
\norm{\bu}_{H^1(\Omega)^d}^2 
\leq \gamma \int_{\Omega} \bu\cdot\bu + \varepsilon(\bu):\varepsilon(\bu) ~\md\bx \quad \forall \, \bu \in V.
\end{equation}
Thus, the convergence rate of the SDC preconditioner will deteriorate for $\nu$
approaching $1/2$, the so-called nearly incompressible case.

We consider the reference configuration $\Omega = (0,1)^d$ discretized on a
Cartesian mesh with 8 cells along each direction. We specify $\mu=1$, a uniform
downwards body force $\bB = -0.02\be_2$, and homogeneous Dirichlet BCs on
$\Gamma_D = \{\bx\in\partial\Omega, x_1=0\}$.  In Table
\ref{tab:primal_elasticity} we present the PCG iteration counts required to
reduce the Euclidean norm of the residual by a factor of $10^8$ starting from a
zero initial guess. As the preconditioner, we employ the hybrid
$p$-multigrid/Schwarz method with vertex-star patches and the SDC/FDM/sparse
relaxation and a coarse space with $p=1$. As expected from \eqref{eq:SDC}, 
the results confirm that the approach is reasonably $p$-robust, but
that robustness with respect to $\nu$ cannot be achieved with SDC relaxation on
vertex-star patches.

\begin{table}[tbhp]
{
\footnotesize
\caption{
PCG iteration counts for the primal formulation of the linear elasticity problem
using the SDC/FDM/sparse relaxation.}
\label{tab:primal_elasticity}
\begin{center}
\rowcolors{2}{gray!25}{white}
\begin{tabular}{rr|*{5}{r}}
\toprule
$d$ & $p\setminus\lambda$ 
& $0$ & $10^0$ & $10^1$ & $10^2$ & $10^3$\\
\midrule
2 & 3 & 13 & 14 & 24 & 70 & 199\\ 
  & 7 & 17 & 17 & 28 & 76 & 236\\ 
  &15 & 18 & 19 & 30 & 81 & 249\\ 
  &31 & 20 & 20 & 32 & 84 & 258\\ 
\midrule
3 & 3& 20 & 22 & 39 & 114 & 362\\  
  & 7& 25 & 28 & 48 & 123 & 381\\  
  &15& 27 & 29 & 51 & 125 & 373\\  
\bottomrule
\end{tabular}
\end{center}
}
\end{table}

\subsection{Mixed FEM formulations of linear elasticity\label{sec:mixed}}

In order to avoid locking in nearly incompressible continua, or impose the
incompressibility constraint, the standard approach is to introduce a
pressure-like variable and discretize with a mixed FEM.  This is expressed by
the weak formulation: find $(\bu-\bu_0, p)\in V\times Q$ such that
\begin{alignat}{2}
a(\bv, \bu) + b(p, \bv) &= L(\bv) && \quad \forall \, \bv \in V,\\
b(q, \bu) - c(q, p) &= 0 && \quad \forall \, q \in Q, \label{eq:divfree}
\end{alignat}
where
\begin{equation}
a(\bv, \bu) = \int_{\Omega} 2\mu \varepsilon(\bv) : \varepsilon(\bu) ~\md\bx, \quad
b(q, \bu) = \int_{\Omega} q ~\div(\bu) ~\md\bx, \quad
c(q, p) = \int_{\Omega}\lambda^{-1} q p ~\md\bx,
\end{equation}
and $Q = L^2_0(\Omega)$ for $\lambda = \infty$ and $\Gamma_D = \partial \Omega$,
or $Q = L^2(\Omega)$ otherwise.

In order for this problem to have a unique solution, we require the well-known
Brezzi conditions: the solution for $\bu$ is unique if $a(\cdot, \cdot)$ is
coercive on the kernel of $b(\cdot,\cdot)$, and the solution for $p$ is unique
if there exists a right inverse for $b(\cdot, \cdot)$.  This is expressed in
the so-called inf--sup condition or LBB condition \cite{babuska73, brezzi74}:
there exists $\beta$, which might depend on $\Omega$, such that
\begin{equation}
\label{eq:infsup}
0<\beta \coloneqq
\adjustlimits \inf_{q\in Q} \sup_{\bv\in V} 
\frac{b(q, \bv)}{a(\bv,\bv)^{1/2} \norm{q}_Q}.
\end{equation}

After selecting suitable finite dimensional subspaces $V_h \subset V$, $Q_h
\subset Q$, we obtain a system of linear equations with the saddle point
structure
\begin{equation}
\label{eq:saddle}
\begin{bmatrix}
A & B^\top\\
B & -C
\end{bmatrix}
\begin{bmatrix}
\ubu\\ \up
\end{bmatrix}
=
\begin{bmatrix}
\ubf\\ \ug
\end{bmatrix}.
\end{equation}

We require the analogous Brezzi conditions for the discrete problem:
that $a(\cdot, \cdot)$ is coercive on
the discrete kernel of $b(\cdot, \cdot)$, and that there exists a discrete
inf-sup constant $\betah$ independent of the mesh but possibly depending on $p$
such that
\begin{equation}
\label{eq:betah}
0<\betah \coloneqq
\adjustlimits \inf_{q_h\in Q_h} \sup_{\bv_h\in V_h} 
\frac{b(q_h, \bv_h)}{a(\bv_h, \bv_h)^{1/2} \norm{q_h}_{Q_h}}.
\end{equation}
The discretization $V_h \times Q_h$ must be chosen carefully to satisfy these
conditions; the discrete inf-sup condition will not be satisfied by arbitrary
discretizations.  If they are, we have the well-known error estimates
\begin{subequations}
\begin{align}
\norm{\bu_h - \bu}_{V} \leq C_1 \left\{ \inf_{\bv_h\in V_h} \norm{\bv_h - \bu}_V  + \inf_{q_h\in Q_h}\norm{q_h - p}_Q\right\},\\
\norm{p_h - p}_{Q} \leq \betah^{-1} C_2 \left\{ \inf_{\bv_h\in V_h} \norm{\bv_h - \bu}_V  + \inf_{q_h\in Q_h}\norm{q_h - p}_Q\right\},
\end{align}
\end{subequations}
where $C_1, C_2>0$ are generic constants independent of the mesh parameter $h$.
For the use of high-order discretizations, it is desirable to choose an element
pair where $\betah$ does not decrease as the polynomial degree of the
approximation is increased. Such a discretization is referred to as $p$-stable.

In fact, $p$-stability is important for solvers also.  Approaches based on
block-Gau{\ss}ian elimination, such as the Uzawa algorithm~\cite{uzawa58} and
block-preconditioned MINRES \cite{paige75}, require preconditioners for the
negative pressure Schur complement $S = C + B A^{-1} B^\top$. It is well known
that for the Stokes system, the continuous analogue of $S$, $\nabla \cdot
(-\nabla^2)^{-1} \nabla$, is well approximated by the identity operator
\cite{silvester94}.  It follows that $S$ is spectrally equivalent to the
pressure mass matrix, $M_p$, 
\begin{equation} \beta_0^2 \leq \frac{\uq^\top
S \uq}{\uq^\top M_p\uq} \leq \beta_1^2 \quad \forall \, \uq \in
\reals^{\dim(Q_h)}\setminus \{0\}.  
\end{equation} 
The rate of convergence of block-preconditioned MINRES will be determined by
the ratio $\beta_1/\beta_0$.  For Stokes flows with pure Dirichlet BCs,
$\beta_1=1$, and $\beta_1 = \sqrt{d}$ otherwise. In general, we have $\beta_0 =
\betah$.  Since $A$ is spectrally equivalent to the vector
Laplacian, these results also hold for linear elasticity. We may expect solvers
based on such techniques to degrade with $p$-refinement if the discretization
is not $p$-stable.

If we choose to work with the $[H^1(\Omega)]^d$-conforming space $V_h =
[\mathrm{Q}_p(\Omega)]^d$, some standard inf--sup stable choices for $Q_h$ are
$\mathrm{Q}_{p-1}(\Omega)$, $\mathrm{DQ}_{p-2}(\Omega)$ and
$\mathrm{DP}_{p-1}(\Omega)$. $\mathrm{DQ}_{p-2}$ denotes discontinuous
piecewise polynomials of degree at most $p-2$ in each direction, while
$\mathrm{DP}_{p-1}$ denotes discontinuous piecewise polynomials of
total degree at most $p-1$. The choice $Q_h = \mathrm{Q}_{p-1}$ gives
rise to the high-order generalization of the Taylor--Hood mixed element
\cite{taylor73}.  Here, $M_p$ will not be block diagonal, and hence more
expensive preconditioning techniques will be required. Moreover, it is shown
numerically in \cite{ainsworth00} that the Taylor--Hood element is not
$p$-stable. The choice $Q_h = \mathrm{DQ}_{p-2}$ exhibits an asymptotic decay
of $\betah \leq C p^{(1-d)/2}$ as $p\to \infty$~\cite{bernardi87}, and thus is
not $p$-stable.  In practice, it is observed that this is quite a pessimistic
bound for moderate $p$ \cite{maday93}.  The choice $Q_h = \mathrm{DP}_{p-1}$ is
$p$-stable, but numerical experiments reveal that the stability is severely
affected by the cell aspect ratio, unlike the previous two choices
\cite{schotzau98}. Moreover this last space does not have tensor product shape
functions, so its efficient implementation becomes challenging.

To construct $p$-stable discretizations that are also robust to cell aspect
ratio, we turn to nonconforming schemes with $V_h \subset H(\div, \Omega)$
\cite{cockburn07, lederer18}.  In particular we consider the use of Raviart--Thomas
elements~\cite{arnold05} of degree $p$ for $V_h$ for the displacement, paired
with $Q_h = \mathrm{DQ}_{p-1}$. This pair satisfies $\div(V_h) = Q_h$, which
enforces the incompressibility constraint \eqref{eq:divfree} exactly in the
numerical approximation for $\lambda = \infty$.  The Raviart--Thomas elements
are defined on the reference quadrilateral as
\begin{equation}
\mathrm{RT}_{p}(\hat{K}) = 
\mathrm{P}_{p}(\hat{\mathcal{I}}) \otimes \mathrm{DP}_{p-1}(\hat{\mathcal{I}}) \oplus 
\mathrm{DP}_{p-1}(\hat{\mathcal{I}}) \otimes \mathrm{P}_{p}(\hat{\mathcal{I}}).
\end{equation}
The analogous element in three dimensions is referred to as the N\'ed\'elec
face element \cite{nedelec80}.  The definition can be extended to curvilinear
cells via the contravariant Piola transform: for a function $\hat{\bu} :
\hat{K} \to \reals^d$, we define $\bu : K \to \reals^d$ as
\begin{equation}
\bu = \mathcal{P}_K(\hat{\bu}) \coloneqq \frac{1}{\abs{\Jac_K}} \Jac_K \left(\hat{\bu} \circ F_K^{-1}\right),
\end{equation}
and set
\begin{equation}
\mathrm{RT}_{p}(K) = \mathcal{P}_K \left(\mathrm{RT}_{p}(\hat{K})\right).
\end{equation}

These elements have superb properties, but their nonconforming nature must be
suitably addressed in the discretization. They
only impose continuity of the normal components of $\bu$ across
cell facets, and we therefore weakly enforce the tangential continuity via the symmetric
interior penalty (SIPG) method~\cite{arnold82}.
The use of SIPG for the displacement requires further extension
of the FDM/sparse relaxation; in particular, we must consider the additional
facet integrals arising in the method, and show that the stiffness remains sparse.

\subsection{Extension to interior penalty DG methods\label{sec:ipdg}}

Interior penalty discontinuous Galerkin (IP-DG) methods relax the continuity
requirement of the discretization space. For instance, instead of $[H^1(\Omega)]^d$, we
consider a larger function space with weaker continuity, such as
$[L^2(\Omega)]^d$ or $H(\div,\Omega)$. As previously mentioned, in order to
deal with the non-conformity, $C^0$-continuity is weakly enforced via the
introduction of a penalty term on the set of interior facets $\Gamma_I$ of the
mesh $\mesh$ that vanishes for $C^0$-continuous functions. Similarly,
the weak prescription of the Dirichlet BC $\bu = \bu_0$ on $\Gamma_D$ is
achieved by introducing a penalty term on $\Gamma_D$.

We consider the following SIPG formulation: 
\begin{equation} \label{eq:sipg_lhs}
\begin{split}
& a(\bv, \bu) = 
\sum_{K\in\mesh}\int_K \nabla \bv : \mathcal{F}^v(\nabla\bu) ~\md \bx\\
   &+ \sum_{e\in \Gamma_I\cup \Gamma_D} \int_{e}
\eta h_e^{-1} \avg{G^\top}\jump{\bv} : \jump{\bu}  
- \jump{\bv} : \avg{\mathcal{F}^v(\nabla \bu)}
- \avg{G^\top \nabla \bv} : \jump{\bu}
~\md s,
\end{split}
\end{equation}
\begin{equation} \label{eq:sipg_rhs}
L(\bv) = \int_\Omega \bv \cdot \bB ~\md\bx
+\int_{\Gamma_D} \eta h_e^{-1} G^\top (\bv \kron \bn) : (\bu_0 \kron \bn)
- G^\top \nabla \bv  : (\bu_0 \kron \bn)  ~\md s.
\end{equation}
Here $\mathcal{F}^v(\nabla\bu)$ is a linear viscous flux.  For the vector
Poisson equation, the viscous flux is given by $\mathcal{F}^v(\nabla\bu) =
\nabla \bu$.  For the primal formulation of linear elasticity, the viscous flux
corresponds to the stress tensor $\mathcal{F}^v(\nabla\bu) = \mu (\nabla\bu +
\nabla\bu^\top) + \lambda\nabla\cdot\bu\eye$.  For the mixed formulation of
linear elasticity, the $(1,1)$-block of the system has viscous flux
$\mathcal{F}^v(\nabla\bu) = \mu (\nabla\bu + \nabla\bu^\top)$.

From left to right, the terms in the surface integral in \eqref{eq:sipg_lhs} are
referred to as the penalty, consistency, and adjoint consistency terms.
The quantity $G$ is known as the homogeneity tensor,
\begin{equation}
G_{ijkl} = \frac{\partial }{\partial u_{k,l}} [\mathcal{F}^v(\nabla \bu)]_{ij},
\end{equation}
for which we define the adjoint product with $\nabla\bv$
\begin{equation}
[G^\top \nabla \bv]_{kl} = G_{ijkl} v_{i,j}.
\end{equation}
The average $\avg{\cdot}$ and jump $\jump{\cdot}$ operators are defined for
scalar, vector, and tensor arguments as follows. Let $e$ be a facet of the
mesh. For an interior facet, let $K^-$ and $K^+$ be the two mesh cells that
share it, and let $\bw^-$ and $\bw^+$ be the traces of a function $\bw$ on $e$
from $K^-$ and $K^+$, respectively. On each facet we define
\begin{equation}
\avg{\bw}_e = 
\begin{cases}
\frac{1}{2}(\bw^- + \bw^+) & e \in \Gamma_I,\\
\bw & \mbox{otherwise},
\end{cases}
\quad
\jump{\bw}_e = 
\begin{cases}
\bw^- \otimes \bn^- + \bw^+ \otimes \bn^+ & e \in \Gamma_I,\\
\bw \otimes \bn & \mbox{otherwise}.
\end{cases}
\end{equation}
In order to ensure coercivity of $a(\cdot,\cdot)$ as we do $h$ or $p$ refinement,
the penalty term must be sufficiently large. 
The penalty coefficient $\eta h_e^{-1}$ must be chosen inversely
proportional to the normal spacing of GLL nodes near the facet, i.e.\
$\eta = \bigo{p(p+1)}$ \cite{manzanero18}. 
For the reciprocal length scale in the direction normal to facet $e$ we use
\begin{equation} 
h_e^{-1} \coloneqq \abs{e} \avg{\abs{K}^{-1}}_e.
\end{equation}

The stiffness matrix that corresponds to $a(\cdot,\cdot)$ in the SIPG formulation 
is obtained via direct stiffness summation over the cells and facets:
\begin{equation}
A = \sum_{K\in \mesh}R_K^\top A^K R_K + \sum_{e\in \Gamma_I \cup \Gamma_D} R_e^\top A^e R_e, 
\end{equation}
where 
$A^K$ is the cell matrix discretizing the volume integral in $K$,
$R_K$ is the Boolean restriction onto the DOFs of $K$,
$A^e$ is the facet matrix discretizing the surface integral on $e$, and
$R_e$ is Boolean restriction onto the DOFs of the cells sharing facet $e$. 

To illustrate the extension of our approach to the SIPG discretization, 
we consider again the scalar Poisson equation. The discrete
problem is to find $u_h \in V_h = \mathrm{DQ}_p(\Omega) \subset L^2(\Omega)$.
On Cartesian cells, both $A^K$ and $A^e$ have a tensor product structure of the
form \eqref{eq:kron}, with matrices of operators on the interval that can be
sparsified by the FDM. To illustrate this, suppose that, for $e\in \Gamma_I$,
$R_e$ reorders the DOFs such that the cells $K^-$ and $K^+$ share $e$ along the
$d$-th reference coordinate axis, while leaving the other axes
consistently oriented on both cells. The facet matrices are
\begin{equation}
A^e = \begin{cases}
E^e \kron \hat{B} & \mbox{if~} d=2,\\
E^e \kron \hat{B} \kron \hat{B} & \mbox{if~} d=3,
\end{cases}
\end{equation}
where the interval facet matrix $E^e$ is defined in terms of the
coefficients $\mu^K_j$ appearing in \eqref{eq:kron}, the 1D shape functions
$\{\hat{\phi}_j\}$, and their normal derivatives $\frac{\partial}{\partial n}
\hat{\phi}_j$ on $\partial \hat{\mathcal{I}}$ (the usual derivative with a
sign).  When $e\in\Gamma_D$, $E^e \in \reals^{(p+1)\times (p+1)}$ is given by
\begin{equation}
\label{eq:dirichlet_stiffness}
[E^e]_{ij} = \mu^e \left( 
\eta \hat{\phi}_i(\hat{x}^e) \hat{\phi}_j(\hat{x}^e)
- \hat{\phi}_i(\hat{x}^e) \tfrac{\partial}{\partial n} \hat{\phi}_j(\hat{x}^e)
- \tfrac{\partial}{\partial n} \hat{\phi}_i(\hat{x}^e) \hat{\phi}_j(\hat{x}^e)
\right).
\end{equation}
Here $\mu^e = \mu^K_l$, where $\hat{x}_l$ is the reference coordinate normal to
$e$, and $\hat{x}^e \in \partial\hat{\mathcal{I}}$ describes the facet $e$ as
the image of the plane $\hat{x}_l = \hat{x}^e$ under $F_K$. When
$e\in\Gamma_I$, $E^e$ is a $2\times 2$ block matrix with blocks $E^e_{rs}\in
\reals^{(p+1)\times (p+1)}$, $r,s \in \{0,1\}$, given by
\begin{equation}
\label{eq:interior_stiffness}
[E^e_{rs}]_{ij} = 
\tfrac{(-1)^{r-s}}{2}
\left(
\eta (\mu^e_0+\mu^e_1) \hat{\phi}_i(\hat{x}^e_r) \hat{\phi}_j(\hat{x}^e_s)
-\mu^e_{s} \hat{\phi}_i(\hat{x}^e_r) \tfrac{\partial}{\partial n}\hat{\phi}_j(\hat{x}^e_s)
-\mu^e_{r} \tfrac{\partial}{\partial n}\hat{\phi}_i(\hat{x}^e_r) \hat{\phi}_j(\hat{x}^e_s)
\right).
\end{equation}
Here $\mu^e_0=\mu^{K^-}_l, \mu^e_1 = \mu^{K^+}_m$, 
where $\hat{x}_l$ and $\hat{x}_m$ are the reference directions normal to $e$ 
on $K^-$ and $K^+$, respectively. Similarly, the facet $e$ is 
the image of $\hat{x}_l = \hat{x}^e_0$ under $F_{K^-}$ and that of
$\hat{x}_m = \hat{x}^e_1$ under $F_{K^+}$.

Some implementations of $\mathrm{DQ}_p$ do not feature an interior-interface
decomposition and use the Gau{\ss}--Legendre (GL) nodal shape functions.  The
GL nodes do not include the endpoints, thus all shape functions have non-zero
support at the facets, causing $E^e$ to be dense.  The matrices $E^e$ are
sparse for a basis with an interior-interface decomposition, such as the GLL
Lagrange polynomials $\{\ell_j\}$, the hierarchical Lobatto polynomials
$\{l_j\}$, and the FDM polynomials $\{\hat{s}_j\}$.  Since $\hat{\phi}_j(\pm
1)$ is non-zero for a single $j\in \{0,p\}$, each term in
\eqref{eq:interior_stiffness} and \eqref{eq:dirichlet_stiffness} corresponds to
a non-zero entry, a non-zero row, and a non-zero column of $E^e$, respectively,
as seen in Figure \ref{fig:ipdg_facet}.

\begin{figure}[H]
\centering
\includegraphics[width=0.24\textwidth,valign=c]{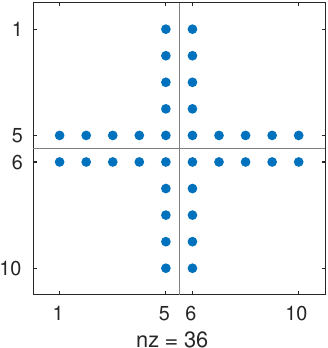}
\caption{
Non-zero structure for the interior facet matrix $E^e$ on the interval ($p=4$).}
\label{fig:ipdg_facet}
\end{figure}

Instead of diagonalizing the SIPG patch matrix as in
\cite{witte20}, our FDM-based approach produces a sparse matrix
with diagonal interior blocks on (possibly) unstructured vertex-star patches.
Figure \ref{fig:spy_ipdg} shows the sparsity pattern of the matrix for the SIPG
formulation of the Poisson equation on a Cartesian vertex-star patch in the FDM
basis, along with its Cholesky factor.  Here the matrix size is increased
from $(2p-1)^d$ DOFs in the CG case to $(2p+2)^d$. At low polynomial
degrees, the interface DOFs form a large fraction of the total number, but the
proportion decreases as $p$ increases. The computational complexity analysis of
Section \ref{sec:complexity} carries over to the SIPG case. 

\begin{figure}[tbhp]
\centering
\subfloat[$A_j$, $d=2$]{
   \includegraphics[height=0.23\textwidth,valign=c]{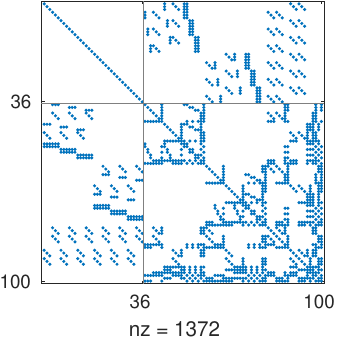}}
\hfill
\subfloat[$\mathrm{chol}(A_j)$, $d=2$]{
   \includegraphics[height=0.23\textwidth,valign=c]{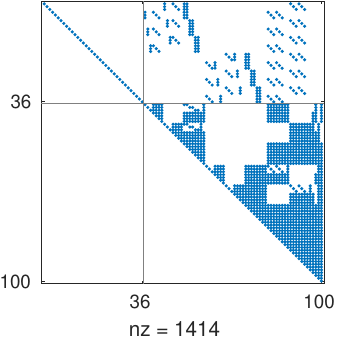}}
\hfill
\subfloat[$A_j$, $d=3$]{
   \includegraphics[height=0.23\textwidth,valign=c]{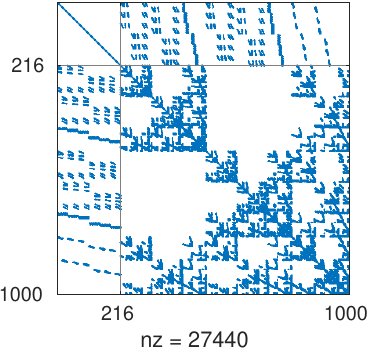}}
\hfill
\subfloat[$\mathrm{chol}(A_j)$, $d=3$]{
   \includegraphics[height=0.23\textwidth,valign=c]{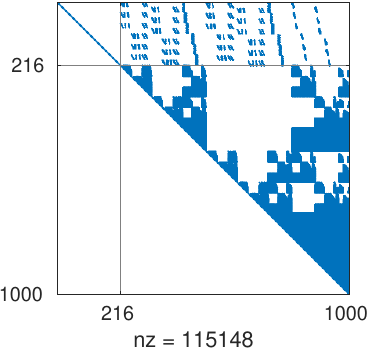}}

\caption{Non-zero structure of the SIPG stiffness matrix in the FDM basis $A_j
= L_jL_j^\top$ and its upper Cholesky factor $L_j^\top$ for the Poisson problem
on a Cartesian vertex-star patch with $p=4$.  Since the space is discontinuous,
the number of DOFs in a patch is increased.  The number of non-zeros in an
interior row is $3d+1$, since the interior DOFs are connected to each of the
$2d$ facets of their corresponding cell, plus $d$ more facets from the adjacent
cells.}
\label{fig:spy_ipdg}
\end{figure}

For a general viscous flux, we construct an auxiliary separable form by
expressing the cell integrals in terms of the reference coordinates
\begin{equation}
a_K(\bv, \bu) = 
\int_K G_{ijkl}^K
\frac{\partial v_i}{\partial x_j} 
\frac{\partial u_k}{\partial x_l} 
~\md\bx =
\int_{\hat{K}} \hat{G}_{ijkl}^K 
\frac{\partial}{\partial \hat{x}_j} \left(v_i\circ F_K \right)
\frac{\partial}{\partial \hat{x}_l} \left(u_k\circ F_K \right)
~\md\hat{\bx},
\end{equation}
where $\hat{G}^K$ is the homogeneity tensor in the reference coordinates,
\begin{equation}
\hat{G}_{ijkl}^K = 
\abs{\det(\Jac_K)}
[\Jac_K^{-1}]_{jm} 
[\Jac_K^{-1}]_{ln} 
G_{imkn}^K.
\end{equation}
The auxiliary form $\tilde{a}(\cdot, \cdot)$ is constructed by approximating $\hat{G}_{ijkl}^K$ with a
piecewise constant tensor that discards the entries where $i\neq k$ or $j\neq l$.  The
corresponding cell stiffness matrices become sparse in the FDM basis, and have a
similar form as \eqref{eq:kron}, except that the coefficients $\mu^K_j$ are
diagonal matrices that multiply each term through an additional
Kronecker product.  Hence, we expect that preconditioners based on the auxiliary form
to be limited by the coupling between vector components, by the mesh geometry,
and by how $G^K$ varies within $K$.

This approach carries over to non-Cartesian cells for DG discretizations of the
Poisson equation in the same way as the CG case.  Unfortunately, the extension
of our approach to vector-valued problems in $H(\div)$ on non-Cartesian cells
does not yield a good preconditioner. In this setting, we construct a block
diagonal preconditioner separating the components of the DOFs, which are in the
reference coordinates.  For the vector Poisson problem on non-Cartesian cells,
the Piola transform introduces off-diagonal contributions from the volume and
surface terms, which does not occur on non-Piola-mapped elements, such as
$[\mathrm{Q}_p]^d$ and $[\mathrm{DQ}_p]^d$.  The excluded terms are required
for the surface integral terms to vanish for arguments with $C^0$ continuity,
and without them the preconditioner might become indefinite on non-Cartesian
cells.

\subsection{Results for mixed formulations of linear elasticity\label{sec:mixed_results}}

We consider the same problem as in Table \ref{tab:primal_elasticity}, with
both a conforming $[\mathrm{Q}_p]^d \times \mathrm{DQ}_{p-2}$ discretization
and a non-conforming $\mathrm{RT}_{p} \times \mathrm{DQ}_{p-1}$ discretization.
For the $H(\div)$-conforming discretization,
the normal components of the Dirichlet BCs are enforced strongly, while
the tangential components of the BCs are weakly enforced with SIPG.
Enforcing the normal conditions strongly is crucial for achieving a divergence-free
solution in the Stokes limit $\lambda = \infty$.
For the penalty coefficient, we use $\eta = p(p+1)$.
We restrict our experiment to Cartesian cells, so that the FDM/sparse relaxation
is applicable to the $H(\div)$-conforming discretization.

We iteratively solve the discrete system $\eqref{eq:saddle}$ 
via MINRES with a symmetric positive definite block diagonal preconditioner,
\begin{equation}
   \mathcal{P}_\mathrm{diag} = 
\begin{bmatrix}
P_1 & 0\\
0 & P_2
\end{bmatrix}.
\end{equation}
Here $P_1$ is a preconditioner for the displacement block $A$, and $P_2$ is a
preconditioner for the scaled pressure mass matrix
$(\mu^{-1}+\lambda^{-1})M_p$.  For $P_1$ we employ the hybrid
$p$-multigrid/Schwarz method with the SDC/FDM/sparse relaxation and 
$[\mathrm{Q}_1]^d$ as the coarse space.  In our
tests, we discretize the pressure space with the GL basis, and 
employ point-Jacobi on the pressure mass matrix,
i.e.\ $P_2=(\mu^{-1}+\lambda^{-1})\diag(M_p)$. When $\mesh$ consists of
Cartesian cells, $M_p = \diag(M_p)$ in the GL basis.  The solver is illustrated
in Figure \ref{fig:solver_minres}.

\begin{figure}[tbhp]
\footnotesize
\centering
\begin{tikzpicture}[%
 every node/.style={draw=black, thick, anchor=west},
grow via three points={one child at (-0.0,-0.7) and
	two children at (0.0,-0.7) and (0.0,-1.4)},
edge from parent path={(\tikzparentnode.210) |- (\tikzchildnode.west)}]
\node {Krylov solver: MINRES}
child {node {Block-diagonal preconditioner $\mathcal{P}_\mathrm{diag}$}
   child {node {(1,1)-block: Hybrid $p$-multigrid/Schwarz V-cycle}
      child {node {Relaxation: SDC/FDM/sparse}}
      child {node {Coarse grid: Cholesky}}
   }
   child[missing]{}
   child[missing]{}
   child {node {(2,2)-block: Mass matrix preconditioner}
      child {node {Relaxation: point-Jacobi}}
   }
};
\end{tikzpicture}
\caption{Solver diagram for the mixed linear elasticity problem.}
\label{fig:solver_minres}
\end{figure}

In Table \ref{tab:minres} we present
MINRES iteration counts for the same configuration considered in Table
\ref{tab:primal_elasticity} in Section \ref{sec:primal}, using the
$[\mathrm{Q}_p]^d \times \mathrm{DQ}_{p-2}$ and $\mathrm{RT}_{p} \times
\mathrm{DQ}_{p-1}$ elements, respectively. Both discretizations yield robust
iteration counts with respect to $\lambda$; the iterations grow with the former
discretization much more quickly than the latter, especially in 3D.

\begin{table}[tbhp]
{
\footnotesize
\caption{
MINRES iteration counts for the mixed linear elasticity problem,
using the solver in Figure \ref{fig:solver_minres}.}
\label{tab:minres}
\begin{center}
\rowcolors{4}{}{gray!25}
\begin{tabular}{rr|*{5}{r}|*{5}{r}}
\toprule
& 
&\multicolumn{5}{c|}{$[\mathrm{Q}_p]^d \times \mathrm{DQ}_{p-2}$} 
&\multicolumn{5}{c}{$\mathrm{RT}_{p} \times \mathrm{DQ}_{p-1}$}
\\ 
$d$ & $p\setminus \lambda$ 
& $10^0$ & $10^1$ & $10^2$ & $10^3$ & $\infty$ 
& $10^0$ & $10^1$ & $10^2$ & $10^3$ & $\infty$ 
\\
\midrule
2 & 3 & 28 & 40 & 43 & 43 & 43 & 25 & 36 & 39 & 40 & 40\\
  & 7 & 31 & 45 & 50 & 51 & 51 & 28 & 40 & 43 & 45 & 45\\
  &15 & 34 & 50 & 57 & 57 & 57 & 30 & 43 & 48 & 48 & 48\\
  &31 & 36 & 53 & 64 & 65 & 65 & 31 & 45 & 51 & 51 & 51\\
\midrule
3 & 3 & 44 & 67 & 75  & 76  & 76  & 34 & 50 & 55 & 56 & 56\\
  & 7 & 50 & 83 & 96  & 97  & 98  & 39 & 58 & 63 & 65 & 65\\
  &15 & 53 & 88 & 111 & 118 & 119 & 41 & 63 & 70 & 70 & 70\\
\bottomrule
\end{tabular}
\end{center}
}
\end{table}

The solver configuration shown in Figure \ref{fig:solver_minres} is optimized
for memory usage, employing a block diagonal preconditioner so that the
short-term recurrences of MINRES may be exploited. If one is willing to trade
memory for time, one may consider an alternative configuration shown in Figure 
\ref{fig:solver_gmres} employing right-preconditioned GMRES \cite{saad86} with a
block upper triangular preconditioner, 
\begin{equation}
\mathcal{P}_\mathrm{upper} = 
\begin{bmatrix}
P_1 & B^\top \\
0 & -P_2 
\end{bmatrix},
\end{equation}
which requires a single application each of $P_1^{-1}$, $P_2^{-1}$, and $B^\top$ per GMRES iteration.
The GMRES iteration counts are presented in Table \ref{tab:gmres}.

\begin{figure}[tbhp]
\footnotesize
\centering
\begin{tikzpicture}[%
 every node/.style={draw=black, thick, anchor=west},
grow via three points={one child at (-0.0,-0.7) and
	two children at (0.0,-0.7) and (0.0,-1.4)},
edge from parent path={(\tikzparentnode.210) |- (\tikzchildnode.west)}]
   \node {Krylov solver: GMRES(30)}
child {node {Block upper triangular preconditioner $\mathcal{P}_\mathrm{upper}$}
   child {node {(1,1)-block: Hybrid $p$-multigrid/Schwarz V-cycle}
      child {node {Relaxation: SDC/FDM/sparse}}
      child {node {Coarse grid: Cholesky}}
   }
   child[missing]{}
   child[missing]{}
   child {node {(2,2)-block: Mass matrix preconditioner}
      child {node {Relaxation: point-Jacobi}}
   }
};
\end{tikzpicture}
\caption{Solver diagram for the mixed linear elasticity problem that trades memory for iteration counts.}
\label{fig:solver_gmres}
\end{figure}

\begin{table}[tbhp]
{
\footnotesize
\caption{ 
GMRES iteration counts for the mixed linear elasticity problem, using the solver in
Figure \ref{fig:solver_gmres}.}
\label{tab:gmres}
\begin{center}
\rowcolors{4}{}{gray!25}
\begin{tabular}{rr|*{5}{r}|*{5}{r}}
\toprule
& 
&\multicolumn{5}{c|}{$[\mathrm{Q}_p]^d \times \mathrm{DQ}_{p-2}$} 
&\multicolumn{5}{c}{$\mathrm{RT}_{p} \times \mathrm{DQ}_{p-1}$}
\\ 
$d$ & $p\setminus \lambda$ 
& $10^0$ & $10^1$ & $10^2$ & $10^3$ & $\infty$ 
& $10^0$ & $10^1$ & $10^2$ & $10^3$ & $\infty$ 
\\
\midrule
2 & 3 & 17 & 23 & 25 & 26 & 26 & 14 & 20 & 22 & 22 & 22\\
  & 7 & 18 & 25 & 27 & 28 & 28 & 16 & 22 & 24 & 24 & 24\\
  &15 & 20 & 27 & 33 & 34 & 34 & 17 & 23 & 26 & 26 & 26\\
  &31 & 21 & 30 & 38 & 38 & 39 & 18 & 24 & 28 & 28 & 28\\
\midrule
3 & 3 & 24 & 32 & 33 & 34 & 34 & 17 & 23 & 25 & 25 & 25\\
  & 7 & 27 & 35 & 38 & 38 & 38 & 21 & 26 & 29 & 29 & 29\\
  &15 & 28 & 38 & 44 & 46 & 46 & 22 & 28 & 31 & 31 & 32\\
\bottomrule
\end{tabular}
\end{center}
}
\end{table}

In Table \ref{tab:beam} we study the performance of our solver on an
unstructured mesh.  We consider the $[\mathrm{Q}_p]^d \times \mathrm{DQ}_{p-2}$
discretization of incompressible linear elasticity ($\lambda = \infty$).  We
prescribe $\mu=1$, a uniform downwards body force $\bB=-0.02\be_2$, and
homogeneous Dirichlet BCs on the displacement on the holes of the domain.  The
three-dimensional mesh is obtained via extrusion by 16 layers of the
two-dimensional mesh.  The iteration counts follow the same pattern as before
for this element: they are not $p$-robust as expected, but they remain modest
even at very high degrees.

\begin{table}[tbhp]
{
\footnotesize
\caption{
GMRES iteration counts for the mixed formulation of the incompressible linear
elasticity problem on the unstructured mesh shown here, using the solver in
Figure \ref{fig:solver_gmres}. 
}
\label{tab:beam}
\begin{center}
\rowcolors{1}{}{gray!25}
\begin{tabular}{rr|*{2}{r}}
\toprule
$d$ & $p$ & \# DOFs & Iter. \\
\midrule
 2 & 3 &    17 466& 26 \\
   & 7 &   104 250& 29 \\
   &15 &   499 002& 35 \\
   &31 & 2 173 242& 44 \\
\midrule
 3 & 3 &   588 927& 39 \\
   & 7 & 7 876 575& 47 \\
   &11 &31 236 927& 51 \\
\bottomrule
\end{tabular}
\quad
\raisebox{-0.5\height}{\includegraphics[width=0.5\linewidth]{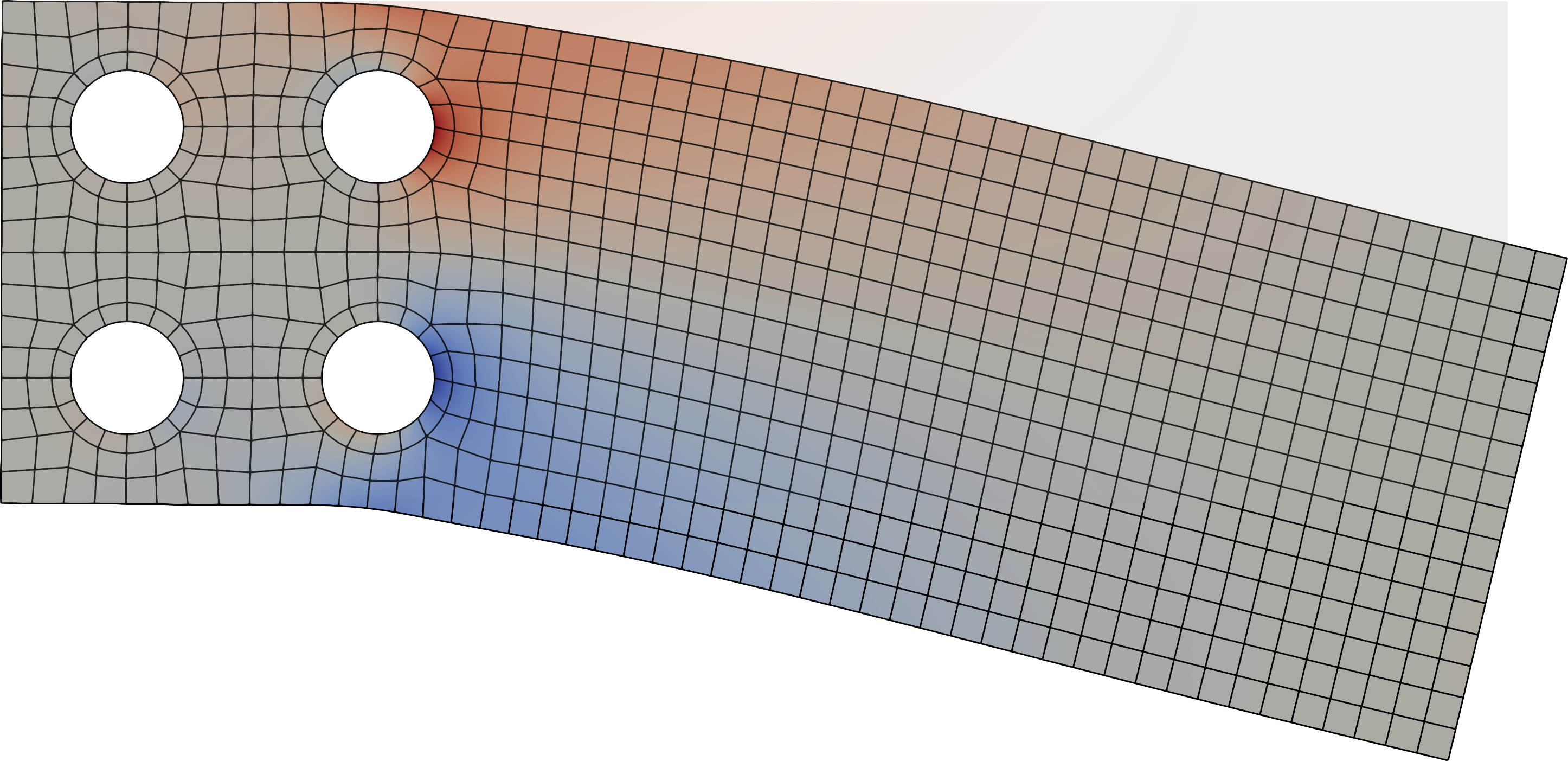}}
\end{center}
}
\end{table}

\section{Conclusion\label{sec:conclusion}}

We have introduced a fast relaxation method required for unstructured
vertex-centered patch problems arising in Pavarino's approach, extending its
practicality to much higher polynomial degrees.  Our method relies on a
spectrally equivalent form, constructed such that it is separable in the
reference coordinates.  We show promising results for the Poisson equation and
mixed formulations of linear elasticity.  A downside of the approach is its
narrow applicability; it will not be effective on more general problems,
especially for those where the dominant terms include mixed derivatives and
mixed vector components. In addition, our method relies on having a good
quality mesh, with its performance depending on the minimal angle; however,
mesh generators with guarantees on the minimal angle are available in two
dimensions~\cite{liang11}.  So far, we have only considered
constant-coefficient problems, but the theory of \cite{axelsson09} suggests
that our approach would remain effective for spatially varying coefficients.
Work in progress shows that the memory bandwidth limitations 
and the suboptimal complexity of the sparse Cholesky
factorization can be overcome by iterative patch solvers, such as
incomplete factorizations and algebraic multigrid.

\bibliographystyle{siamplain}
\bibliography{references}
\end{document}